\documentclass[10pt, reqno]{amsart}
\usepackage{hyperref, url}
\hypersetup{
	colorlinks=true,
	citecolor=blue,
	linkcolor=blue,
	filecolor=magenta,      
	urlcolor=cyan,
}

\usepackage[capitalize]{cleveref}
\crefname{equation}{}{}
\usepackage{fullpage}
\usepackage{setspace}
\usepackage{amsmath, amsthm, amssymb, graphicx, dsfont, stmaryrd, extarrows, enumitem}
\usepackage[T1]{fontenc}
\usepackage{textcomp}
\usepackage{lmodern}
\usepackage{microtype} 
\usepackage{amscd}
\usepackage{mathrsfs}
\usepackage{tikz-cd}
\usetikzlibrary{matrix,patterns}
\usepackage[colorinlistoftodos]{todonotes}
\usepackage{color}
\usepackage{bbm}
\usepackage{verbatim}
\usepackage[mathscr]{euscript}
\usepackage{bm}

\numberwithin{equation}{section}
\theoremstyle{plain}
\newtheorem{thm}[equation]{Theorem}
\newtheorem*{thm*}{Theorem}

\newtheorem{prop}[equation]{Proposition}
    
\newtheorem{cor}[equation]{Corollary}       
\newtheorem{lem}[equation]{Lemma}

\theoremstyle{definition} 
\newtheorem{defn}[equation]{Definition} 

\newtheorem{ex}[equation]{Example}

\newtheorem{rem}[equation]{Remark}

\newcommand{\Z}{\mathbb{Z}}
\newcommand{\Q}{\mathbb{Q}}
\newcommand{\qz}{\Q/\Z}
\newcommand{\Hom}{\mathrm{Hom}}

\newcommand{\msf}[1]{\mathsf{#1}}

\newcommand{\mc}[1]{\mathcal{#1}}
\newcommand{\mrm}[1]{\mathrm{#1}}
\newcommand{\mbb}[1]{\mathbb{#1}}

\newcommand{\p}{\mathfrak{p}}

\newcommand{\T}{\mathsf{T}}
\newcommand{\U}{\mathsf{U}}

\newcommand{\1}{\mathbbm{1}}

\renewcommand{\mod}[1]{\mathrm{Mod}_{#1}}

\newcommand{\A}{\mathsf{A}}
\newcommand{\B}{\mathsf{B}}
\newcommand{\C}{\mathscr{C}}
\newcommand{\D}{\mathsf{D}}
\newcommand{\G}{\mathscr{G}}
\newcommand{\K}{\mathscr{K}}
\newcommand{\bctwo}[2]{\mathbb{I}_{#1}({#2})}
\newcommand{\bc}[1]{\mathbb{I}{#1}}

\renewcommand{\mod}[1]{\mrm{mod}(#1)}
\newcommand{\Mod}[1]{\mrm{Mod}(#1)}
\newcommand{\ab}{\mathbf{Ab}}
\renewcommand{\t}{\text}
\newcommand{\ti}{\textit}
\renewcommand{\tilde}[1]{\widetilde{#1}}
\newcommand{\rlim}{\varinjlim}

\renewcommand{\t}{\text}
\renewcommand{\ti}{\textit}
\newcommand{\tb}{\textbf}
\newcommand{\mf}{\mathfrak}
\renewcommand{\c}{\mathrm{c}}
\newcommand{\X}{\msf{X}}
\newcommand{\Y}{\msf{Y}}

\makeatletter
\@namedef{subjclassname@2020}{\textup{2020} Mathematics Subject Classification}
\makeatother
\begin{document}
	
\title{Duality pairs, phantom maps, and definability in triangulated categories}
\author{Isaac Bird}
\address[Bird]{Department of Algebra, Faculty of Mathematics and Physics, Charles University in Prague, Sokolovsk\'{a} 83, 186 75 Praha, Czech Republic}
\email{bird@karlin.mff.cuni.cz}
\author{Jordan Williamson}
\address[Williamson]{Department of Algebra, Faculty of Mathematics and Physics, Charles University in Prague, Sokolovsk\'{a} 83, 186 75 Praha, Czech Republic}
\email{williamson@karlin.mff.cuni.cz}
\begin{abstract}
	We define duality triples and duality pairs in compactly generated triangulated categories and investigate their properties. This enables us to give an elementary way to determine whether a class is closed under pure subobjects, pure quotients and pure extensions, as well as providing a way to show the existence of approximations. One key ingredient is a new characterisation of phantom maps. 
	We then introduce an axiomatic form of Auslander--Gruson--Jensen duality, from which we define dual definable categories, and show that these coincide with symmetric coproduct closed duality pairs. 
	This framework is ubiquitous, encompassing both algebraic triangulated categories and stable homotopy theories. Accordingly, we provide many applications in both settings, with a particular emphasis on silting theory and stratified tensor-triangulated categories.

\end{abstract}
\maketitle
\setcounter{tocdepth}{1}
\tableofcontents

\section{Introduction}
Since their introduction, triangulated categories have become a fundamental object of study in algebra, geometry and topology, where they encode all the homological and homotopic information of the objects under investigation. Their ubiquity is seen through examples such as the derived category of modules over a ring, the derived category of (quasi)coherent sheaves on a scheme, and the homotopy category of spectra, amongst a plethora of others. In fact, many examples of interest (including those recalled above) are instances of \emph{compactly generated} triangulated categories - ones in which there is a set of `small' objects, called \emph{compact} objects, which generate the entire category through the triangulated structure.

The property of being compactly generated endows the triangulated category with many pleasant properties. One such property is the existence of a universal functor to a Grothendieck abelian category which sends triangles to long exact sequences, and preserves coproducts. This Grothendieck abelian category is actually nothing other than the category of additive contravariant functors from the compact objects to abelian groups, and the associated functor is just the (restricted) Yoneda embedding. If one is only interested in the triangles which get sent to short exact sequences under this functor, rather than long exact sequences, one picks out the \ti{pure triangles}, and, using these, one can consider the pure structure of a compactly generated triangulated category. 

Purity has long been a staple for understanding and describing the structure of Grothendieck abelian, and particularly module, categories, as well as a vital device in answering many representation theoretic questions, such as classification problems for finite dimensional algebras and (co)tilting theory. A wide range of such applications, along with further references, may be found in~\cite{gt, jl, psl} for instance. The ability to transfer the utility of purity from the abelian functor category to the compactly generated triangulated setting means purity has emerged as a vital tool in many approaches to understanding the structure of compactly generated triangulated categories. For instance, several fundamental questions in homological algebra, representation theory, and homotopy theory can be considered through the lens of purity in such triangulated categories. We provide some examples as motivation for our study. 

The telescope conjecture asks whether the kernel of every smashing localisation is generated by a set of compact objects. Krause~\cite{krsmash} used purity as a central tool to reframe the conjecture in an algebraic form and then prove a modified version of it. This conjecture was recently disproved for the stable homotopy category of spectra~\cite{telescope}, but it remains interesting to study the conjecture in other compactly generated triangulated categories. Margolis's uniqueness conjecture for the stable homotopy category is another significant open problem. It asks whether the finite spectra determine the whole stable homotopy category, that is, whenever $\T$ is a compactly generated triangulated category with compact objects $\T^{\c}$ such that $\T^{\c}$ is equivalent to the category of finite spectra, then $\T$ is equivalent to the category of spectra. It was shown in~\cite{phantom} that this conjecture can be reduced, through the lens of purity, to a more elementary equivalent question. 

On the more algebraic side, since their introduction in ~\cite{bbd} where they were used to define perverse sheaves, $t$-structures have become a fundamental tool in understanding derived categories of sheaves and modules. The existence of $t$-structures is dependent on the existence of certain approximations or adjoints. The existence of approximations by a class of objects can in turn be determined through checking whether that class is closed under certain properties, often pure subobjects or pure quotients~\cite{covers, krauseapprox}. We will return to this as an application of the framework we develop later on, see \cref{dp}. In representation theory, $t$-structures and torsion pairs play a key role in the study of silting objects~\cite{AMV, MarksVitoria}, which are a generalisation of tilting objects as considered by Rickard~\cite{Rickard}. There are different motivations for studying silting objects: firstly, they are more flexible than tilting objects when it comes to mutation - one can obtain a new silting object from an old one \cite{alsv}. Secondly, silting objects can be seen as a generalisation of injective modules to triangulated categories. We will consider this perspective in the main body of the paper.

One fundamental property of purity in module categories, and more generally Grothendieck abelian categories, is that a pure exact sequence is, in some sense, a weak version of a split exact sequence: a sequence is pure exact if and only if its image under the character dual $\Hom_{\Z}(-,\qz)$ splits. However, unlike checking if a class is closed under direct summands, checking if a class has desirable purity closure properties is usually a difficult problem. For categories of modules over a ring, a useful concept, called a \emph{duality pair}, was introduced by Holm and J\o rgensen in~\cite{hj} to determine pure closure properties. Such a pair consists of two classes of modules $(\A,\B)$ such that $M \in \A$ if and only if $\Hom_\mbb{Z}(M,\qz) \in \B$, and $\B$ is closed under summands and finite direct sums. They used this splitting property of pure exact sequences to prove that if $(\A,\B)$ is a duality pair then $\A$ has good pure closure properties, and frequently yields approximations.

In this paper we provide an axiomatic framework for detecting pure triangles in compactly generated triangulated categories through an exact functor that generalises the character dual functor from the Grothendieck abelian setting. This enables us to define duality pairs, thus providing a scaffold from which one can verify pure closure properties. Moreover, we illustrate how duality pairs can, as suggested by the prior discussion, be used to determine the existence of approximations. 

These results are widely applicable: all algebraic triangulated categories, as well as all rigidly-compactly generated tensor triangulated categories, fall within the framework. Following this abstract setup, we provide many examples which arise naturally in algebra and topology, before turning our attention to applications, which predominantly lie in representation theory and stable homotopy theory. We also consider the relationship between the duality pairs we define, and those defined in~\cite{hj}, as well as the well established relationship between a compactly generated triangulated category, and its functor category. 

A further motivation for us to develop the framework enabling us to construct duality pairs in the triangulated setting is the goal of better understanding definable subcategories. These are the subcategories which are determined by the vanishing of coherent functors, or equivalently in the presence of an enhancement, by closure under pure subobjects, products, and directed homotopy colimits. We emphasise that definable subcategories need not be triangulated. 

Definable subcategories have long played an important role in representation theory, and understanding module categories~\cite{psl, dac}. However, it has been shown that they are of equal utility in the triangulated setting: for example, from an algebraic standpoint, every silting class is definable, as are Auslander and Bass classes of complexes, and many other classes of complexes of finite (co)homological dimension. Further motivation from the topological viewpoint is Freyd's generating hypothesis, which conjectures that a map $f$ between finite spectra is null homotopic if and only if the induced map on homotopy $\pi_*(f)$ is zero. This can be formulated in terms of definability~\cite{krcoh}, and as such the techniques of purity may provide new insight, just as in the seminal work of Krause on the telescope conjecture~\cite{krsmash}. 

We now discuss the contents and main results of this paper in more detail. In order to construct an axiomatic setting in which we can consider duality pairs, we first need a functor abstracting the character dual with the aim of connecting pure and split triangles as in the module case. By distilling the salient features of the character dual, we define a duality triple $(\T, \msf{U}, Q)$; see \cref{dualitytriple}. This consists of a pair of compactly generated triangulated categories $\T$ and $\msf{U}$ and exact functors $Q\colon \T^\t{op} \to \msf{U}$ and $Q\colon \msf{U}^\t{op} \to \T$ satisfying various natural properties which mirror those in the Grothendieck abelian setting. Our first main result shows that the structure of a duality triple completely determines the pure structure on a triangulated category. 

\begin{thm*}[\cref{detects}]
Let $(\T,\msf{U},Q)$ be a duality triple. Then a triangle $A\to B\to C$ in $\T$ is pure if and only if the induced triangle $Q(C)\to Q(B)\to Q(A)$ splits. Equivalently, $f\colon X\to Y$ is phantom if and only if $Q(f)=0$.
\end{thm*}

We show that almost all compactly generated triangulated categories that arise in algebraic, geometric, and topological settings admit the structure of a duality triple. Thus the previous theorem is widely applicable: we prove in \cref{ttdualitytriple} that any rigidly-compactly generated tensor-triangulated category $\T$ has the natural structure of a duality triple $(\T,\T,\mathbb{I})$ where $\mathbb{I}$ is the Brown--Comenetz dual. We also show that any compactly generated algebraic triangulated category has the natural structure of a duality triple via a functor induced by the character dual, see \cref{algebraicdualitytriple}. Therefore the above theorem has applications to a broad range of categories, such as the homotopy category of spectra, the derived category of a ring, quasicoherent sheaves on a scheme, stable categories of Frobenius categories, and certain categories of motives.

If one is given a duality triple $(\T, \msf{U}, Q)$ one may define a duality pair $(\A,\B)$, where $\A\subset\T$ and $\B\subset\msf{U}$ in the same way as in ~\cite{hj}: the pair $(\A,\B)$ is a duality pair provided $\B$ is closed under summands and finite direct sums, and $A\in\A$ if and only if $Q(A)\in\B$. Using this definition, we are able to apply \cref{detects} to show that, as in the case for modules over a ring, duality pairs in the triangulated setting also admit powerful pure closure properties.

\begin{thm*}[\cref{dp}]
If $(\A,\B)$ is a duality pair, then the class $\A$ is closed under pure subobjects, pure quotients and pure extensions.
\end{thm*}

We combine the preceding theorem with results from~\cite{lv}, to obtain criteria for when $\A$ provides approximations, or appears as a torsion class. Consequently as claimed, duality pairs can be a useful tool in understanding inclusions of subcategories and the existence of adjoints.

As mentioned above, one of our motivations for studying duality pairs is to provide a better understanding of definable subcategories, and we now turn our attention to this problem. Definable subcategories are determined by finitely presented functors, and therefore any interplay between duality pairs and definable categories must entail some way to put a duality on finitely presented functors. We provide such a duality by introducing the refined concept of an \emph{Auslander--Gruson--Jensen duality triple}; despite this refinement, the main examples of duality triples we consider in this paper satisfy this additional requirement. We show in \cref{agj} that this refined structure is precisely what is needed to establish an exact antiequivalence $\delta \colon (\T^{\c},\ab)^{\t{fp}}\to(\U^{\c},\ab)^{\t{fp}}$ between the categories of finitely presented additive functors on $\T^{\c}$ and $\U^{\c}$. Using $\delta$, and the bijective correspondence between definable subcategories of $\T$ and Serre subcategories of $(\T^{\c},\ab)$, we define the \ti{dual definable subcategory}, denoted $\mc{D}^{d}\subset\msf{U}$, of a definable subcategory $\mc{D}$ of $\T$. One would at this point hope that the dual definable subcategory should be related to duality pairs; the following theorem shows that this is indeed the case, and moreover that definable classes are completely characterised via duality pairs. 

\begin{thm*}[\cref{dualdefinable},~\cref{definablett}]
Let $\mc{D}$ be a definable subcategory. Then $(\mc{D}, \mc{D}^d)$ is a symmetric duality pair. Moreover, if $(\A,\B)$ is a symmetric duality pair, then the following are equivalent:
\begin{enumerate}
\item either class is closed under coproducts;
\item either class is closed under products;
\item $\A$ is definable and $\B = \A^d$;
\item $\B$ is definable and $\A = \B^d$.
\end{enumerate}
\end{thm*}

The above theorem is also shown to hold in the setting of Grothendieck abelian categories, where it is also new. Both results completely characterise definable subcategories in terms of symmetric duality pairs and closure under a single property - either coproducts or products. Although the above theorem makes no mention of triangulated subcategories, we obtain the following corollaries: firstly, a definable category $\mc{D}$ is closed under triangles if and only if its dual definable category is (\cref{triiffdual}), and secondly, if $\T$ has a compatible tensor product, then $\mc{D}$ is a $\otimes$-ideal if and only if its dual definable category is (\cref{idealdefinable}). We note that these corollaries were previously observed in~\cite{wagstaffe} using completely independent machinery, but our techniques provide substantially more elementary and brief proofs. 

Having described the main abstract results and tools, we turn to explaining the applications we obtain. The first of these is in representation theory, where we consider the relationship between duality pairs and silting theory, see \cref{sec:silting}. The main result of this section is the following:
\begin{thm*}[\cref{dualiscosilting},~\cref{siltingthm}]
Let $R$ be a ring, and let $X$ be a bounded silting object in $\msf{D}(R)$. Then $X^+ := \msf{R}\Hom_\Z(X,\qz)$ is a bounded cosilting object in $\msf{D}(R^\circ)$. Moreover, $(\msf{Silt}(X), \msf{Cosilt}(X^+))$ is a symmetric duality pair, $\msf{Silt}(X)^d = \msf{Cosilt}(X^+)$ and $\msf{Cosilt}(X^+)^d = \msf{Silt}(X)$.
\end{thm*}

This result is a generalisation of a well established result about the duality between tilting and cotilting modules, and their induced classes, in module categories. It is well known that silting and cosilting class are definable~\cite{MarksVitoria}, and as such the contribution of the previous result is twofold. Firstly, it shows that the character dual on the derived category passes to a duality on silting and cosilting classes recovering a result of~\cite{HrbekHugel}, and secondly, that moreover, this duality corresponds to the duality on definable subcategories.

Our second major application is in the study of tensor-triangulated categories. Many triangulated categories in nature have the additional structure of a compatible monoidal structure, and this monoidal structure often opens the door to new approaches of study. For example, a scheme cannot be reconstructed from the triangulated structure on its derived category alone, but can be reconstructed from the \emph{tensor}-triangular structure~\cite{balmerspectrum}. As such, the study of tensor-triangulated categories is a generalisation of algebraic geometry, and algebro-geometric techniques and constructions have now become a vital part of the toolkit. For example, just as is standard in algebraic geometry, the global structure of a tensor-triangulated category can often be determined via performing localisations and then gluing together the results. 

Therefore we often seek to understand the global structure of a category via the behaviour of certain types of subcategories. For example, there is a set of thick tensor ideals of compacts so we can understand the structure of the compact objects via this, but to understand big objects we have to go further. The natural extension of this to the big setting, is that of localising tensor ideals, however, it is not known whether there is a set of these. On the other hand, there is only a set of definable subcategories, so they provide a way to extend the paradigm from small objects to big objects. Moreover, definable classes also encompass many subcategories of substantial interest: for example, the set of smashing subcategories is in bijection with the set of triangulated definable subcategories. 

Just as the Zariski spectrum of a commutative ring $R$ provides a local-to-global approach for understanding the structure and properties of $R$-modules, there is a categorification of this to tensor-triangulated categories. In fact, associated to a tensor-triangulated category, there are many `spectra' of interest, each isolating a particular feature. For example, the Balmer spectrum determines the collection of thick tensor ideals~\cite{balmerspectrum}, and the Ziegler spectrum determines the definable subcategories~\cite{birdwilliamsonhomological, psl}.  
In this vein, our main application to tensor-triangulated categories determines the interaction between localisations and definability, and hence implicitly between localisation and the Ziegler spectrum.
\begin{thm*}[\cref{ttselfdual},~\cref{stratequiv}]
Let $\T$ be a rigidly-compactly generated tensor-triangulated category. 
\begin{enumerate}
\item If $L$ is a smashing localisation and $\mc{D}$ is a definable subcategory of $\T$, then $L\mc{D}$ is a definable subcategory of $L\T$ and $(L\mc{D})^d = L(\mc{D}^d)$. 
\item Moreover, if $\T$ is stratified in the sense of~\cite{bik2}, then a localising $\otimes$-ideal is product closed if and only if it is self-dual definable. In particular, if $L$ is a smashing localisation then $L\T$ is self-dual definable, and any definable triangulated $\otimes$-ideal in $\T$ is self-dual.
\end{enumerate}
\end{thm*}

As previously discussed, purity in a compactly generated triangulated category $\T$ is determined by purity in the category $\Mod{\T^{\c}}$ of additive functors $(\T^{\mrm{c}})^{\mrm{op}}\to\ab$. Since duality pairs concern purity, one may naturally ask whether duality pairs in $\T$ can be related to duality pairs in $\Mod{\T^{\c}}$. This is the main study of \cref{dpinducedbyfunctors}, where we show that by using a nerve construction one can construct an appropriate structure to define duality pairs on the corresponding functor categories. By considering an assignment $(-)_0$ which produces a class in $\T$ from a class in $\Mod{\T^{\c}}$ by restricting to representable functors, we are able to show that all duality classes in an Auslander--Gruson--Jensen duality triple come from a duality class in a functor category. Moreover, this assignment refines to duality pairs that are symmetric, and even definable. The main results are summarised in the following theorem.

\begin{thm*}[\cref{symmlift},~\cref{defsymlift},~\cref{dplift}]
Let $(\T,\U,Q)$ be an Auslander--Gruson--Jensen duality triple, and let $(\X,\Y)$ be a duality pair on $(\T,\U,Q)$.
\begin{enumerate}
\item There exists a duality pair $(\A,\B)$ where $\A \subseteq \Mod{\T^\c}$ and $\B \subseteq \Mod{\U^\c}$, such that $\A_0 = \X$ and $\B_0 \subseteq \Y$. 
\item If $(\X,\Y)$ is a symmetric duality pair, then there exists a symmetric duality pair $(\A,\B)$ where $\A \subseteq \Mod{\T^\c}$ and $\B \subseteq \Mod{\U^\c}$, such that $\A_0 = \X$ and $\B_0 = \Y$. Moreover, if $\X$ and $\Y$ are definable then so are $\A$ and $\B$.
\end{enumerate}
\end{thm*}
 
In particular, as an application, we show in \cref{ex:moduledualitypair} that every duality pair in a category of modules can be recovered from a duality pair on the derived category. Consequently, the framework developed in this paper extends and completely subsumes the classical concept introduced in \cite{hj}.

\subsection*{Acknowledgements}
We are grateful to Rose Wagstaffe for sharing a copy of her PhD thesis with us, and for helpful discussions. We thank Scott Balchin, Sebastian Opper, Luca Pol, and Alexandra Zvonareva for their comments on a preliminary version of this paper. We also thank Rosanna Laking, Jan \v{S}\v{t}ov\'{\i}\v{c}ek, and Jorge Vit\'oria for helpful discussions and comments regarding silting. We thank the anonymous referee for their helpful comments. Both authors were supported by the grant GA~\v{C}R 20-02760Y from the Czech Science Foundation. 

\section{Purity and duality pairs in module categories}
In this section we provide some background material on purity, as well as recall duality pairs, in the category of modules over a ring. The aim of this is to provide motivation for our study of analogous concepts in the triangulated setting, as well as to highlight the similarities between purity in compactly generated triangulated categories and module categories. 

Let $R$ denote a unital associative ring. We denote the category of left $R$-modules by $\Mod{R}$, and the category of finitely presented left $R$-modules by $\mod{R}$. We write $\Mod{R^{\circ}}$ for the category of right $R$-modules, viewed as left modules over the opposite ring $R^{\circ}$. We will let $\tb{Ab}$ denote $\Mod{\Z}$, the category of abelian groups.

\begin{thm}[{\cite[6.4]{jl}}]\label{puredef}
The following are equivalent for a short exact sequence $S\colon 0\to L\to M\to N\to 0$ in $\Mod{R}$:

\begin{enumerate}
\item for each finitely presented $A\in\Mod{R}$, the induced sequence
\[
0\to\Hom_{R}(A,L)\to\Hom_{R}(A,M)\to\Hom_{R}(A,N)\to 0
\]
is exact in $\ab$;
\item for every (finitely presented) $B\in\Mod{R^{\circ}}$, the induced sequence
\[
0\to B\otimes_{R}L\to B\otimes_{R}M\to B\otimes_{R}N\to 0
\]
is exact in $\ab$;
\item there is a directed system $\mc{S} = (0\to L_{i}\to M_{i}\to N_{i}\to 0)_{i\leq \kappa}$ of split exact sequences of finitely presented modules such that $\rlim_{i\leq \kappa}\mc{S}=S$;
\item
the induced sequence 
\[
0\to \Hom_{\Z}(N,\qz)\to\Hom_{\Z}(M,\qz)\to\Hom_{\Z}(L,\qz)\to 0
\]
is split in $\Mod{R^{\circ}}$.
\end{enumerate}
\end{thm}

\begin{defn}
A short exact sequence is called $\ti{pure}$ if it satisfies any of the equivalent conditions given in \cref{puredef}. We call a monomorphism $L\to M$ in $\Mod{R}$ \ti{pure} if it appears within a pure exact sequence, and we say that $L$ is a \ti{pure submodule} of $M$. The notions of \ti{pure epimorphisms} and \ti{pure quotients} are similarly defined. We will say that a class $\mc{C}$ is closed under \ti{pure extensions} if, given a pure exact sequence $0\to L\to M\to N\to 0$ with $L,N\in\mc{C}$, we also have $M\in\mc{C}$.
\end{defn}

\begin{rem}
The equivalences of the statements in \cref{puredef} are specific to the category $\Mod{R}$. Indeed, depending on the properties of the category at hand, one can define different purity theories. Most general (of the ones we consider) is $\lambda$-purity in accessible categories, as detailed in \cite{ar1} which is sometimes referred to as categorical purity for $\lambda=\aleph_{0}$. For locally finitely presented abelian categories $\mc{A}$, categorical purity is equivalent to the analogues in $\mc{A}$ of \cref{puredef}(1) and \cref{puredef}(3). However, if $\mc{A}$ is equipped with a monioidal structure, these equivalences in general fail to be equivalent to the analogue of \cref{puredef}(2). Indeed, if $X$ is a quasiseparated and quasicompact scheme, then the tensor product definition of purity in $\t{QCoh}(X)$ coincides with categorical purity if and only if $X$ is an affine scheme, as detailed in \cite{sheaves}.
\end{rem}

Having introduced purity, we now recall the definition of a duality pair of modules, as originally presented by Holm and J\o rgensen in \cite{hj}.

\begin{defn}[{\cite[2.1]{hj}}]
Let $R$ be a ring. A pair of classes $(\msf{A},\msf{B})$ of modules with $\msf{A}\subseteq\Mod{R}$ and $\msf{B}\subseteq\Mod{R^{\circ}}$ is a \ti{duality pair} if the following two conditions hold:
\begin{enumerate}
	\item $M\in\msf{A}$ if and only if $\Hom_{\Z}(M,\qz)\in\msf{B}$;
	\item $\msf{B}$ is closed under finite direct sums and direct summands.
\end{enumerate}
A duality pair $(\msf{A},\msf{B})$ is called:
\begin{itemize}
	\item \ti{(co)product closed} if $\msf{A}$ is closed under (co)products in $\Mod{R}$;
	\item \ti{perfect} if $\msf{A}$ is coproduct closed, extension closed and contains $R$.
\end{itemize}
\end{defn}

There is an abundance of duality pairs, but possibly the canonical example is the pair $(\msf{F}_{R},\msf{I}_{R^{\circ}})$ comprising of flat $R$-modules and injective $R^{\circ}$-modules. This is a coproduct closed duality pair over any ring, and is product closed if and only if $R$ is right coherent.

Before stating Holm and J\o rgensen's main theorem concerning duality pairs, we recall some definitions from approximation theory. If $\mc{F}$ is a class of modules and $M$ is an $R$-module, a morphism $f\colon F\to M$ with $F \in \mc{F}$, is an \ti{$\mc{F}$-precover} if for any $g\colon F'\to M$ with $F'\in\mc{F}$ there is a map $\alpha\colon F'\to F$ such that $g=f\alpha$. If $f\colon F\to M$ is an $\mc{F}$-precover, we say it is an \ti{$\mc{F}$-cover} if for any $\alpha\colon F\to F$ such that $f=f\alpha$, we have $\alpha\in\t{Aut}(F)$. The categorically dual notions are called (pre)envelopes. 

Given a class $\mc{A}$ of $R$-modules, we let
\[\mc{A}^{\perp}=\{M\in\Mod{R}:\t{Ext}_{R}^{1}(A,M)=0\t{ for all }A\in\mc{A}\},\]
and similarly define $^{\perp}\mc{A}$. A pair of classes $(\msf{X},\msf{Y})$ of $\Mod{R}$ is called a \ti{cotorsion pair} if $\msf{X}^{\perp}=\msf{Y}$ and $^{\perp}\msf{Y}=\msf{X}$. A cotorsion pair $(\msf{X},\msf{Y})$ is said to be \ti{perfect} if $\msf{X}$ is covering and $\msf{Y}$ is enveloping.

The following is the main theorem from \cite{hj}.

\begin{thm}[{\cite[3.1]{hj}}]\label{hjtheorem}
	Let $(\msf{A},\msf{B})$ be a duality pair. Then $\msf{A}$ is closed under pure submodules, pure quotients and pure extensions. Moreover:
	
	\begin{enumerate}
		\item if $(\msf{A},\msf{B})$ is coproduct closed, then $\msf{A}$ is covering;
		\item if $(\msf{A},\msf{B})$ is product closed, then $\msf{A}$ is preenveloping;
		\item if $(\msf{A},\msf{B})$ is perfect, then $(\msf{A},\msf{A}^{\perp})$ is a perfect cotorsion pair.
		
	\end{enumerate}
\end{thm}

We will include the proof of the non-enumerated statement for illustrative reasons, as it will become relevant and provide motivation for the subsequent discussion in the triangulated setting.

\begin{proof}
Let $0\to L\to M \to N\to 0$ be a pure exact sequence, so $0\to \Hom_{\Z}(N,\qz)\to\Hom_{\Z}(M,\qz)\to\Hom_{\Z}(L,\qz)\to 0$ is split by \cref{puredef}. If $M\in\msf{A}$, then $\Hom_{\Z}(M,\qz)\in\msf{B}$, and thus, as $\msf{B}$ is closed under summands, both $\Hom_{\Z}(L,\qz)$ and $\Hom_{\Z}(N,\qz)$ are in $\msf{B}$. Yet this means $L$ and $N$ are in $\msf{A}$, hence $\msf{A}$ is closed under pure submodules and quotients. The statement about pure extensions follows from similar reasoning.
\end{proof}

Through the above theorem, duality pairs have become a useful tool in determining both closure and approximation properties for classes of modules, such as for example, covers and envelopes by Gorenstein injective modules in \cite{ei}. More recently, interplays between duality pairs and model structures on the category of chain complexes of $R$-modules have been developed, such as in \cite{g1} or \cite{gi}.

\begin{rem}
In the definition of a duality pair, it is possible to replace the functor $\Hom_{\Z}(-,\qz)$ with an appropriate substitute. If $R$ is an $S$-algebra for some commutative ring $S$, and $E$ is an injective cogenerator for $S$-modules, then, instead of using the character duality $\Hom_{\Z}(-,\qz)$, the functor $\Hom_{S}(-,E)$ can be used, and all results written above still hold~\cite[2.19]{gt}. One can also wonder if a pair of classes $(\msf{A},\msf{B})$ can be a duality pair with respect to one duality, and not with respect to another. This can be the case, but only up to products, as stated in \cite[3.2]{mehdi}. Thus, if the class $\msf{B}$ is closed under products, $(\msf{A},\msf{B})$ will be a duality pair for all dualities if it is a duality pair with respect to $\Hom_{\Z}(-,\qz)$. 
\end{rem}

\section{Duality triples and phantom maps}
Throughout $\msf{T}$ will denote a triangulated category which has arbitrary set-indexed coproducts, with shift functor $\Sigma$. Recall that an object $C\in\msf{T}$ is called \ti{compact} if the natural map
\[
\bigoplus_{I}\Hom_{\msf{T}}(C,X_{i}) \to \Hom_{\T}(C,\bigoplus_{I}X_{i})
\]
is an isomorphism for any set $\{X_{i}\}_{I}$ of objects of $\msf{T}$. For a compact object $C$ and any object $X \in \T$ we write $H_i^C(X) = \Hom_\T(\Sigma^i C, X)$. The category $\msf{T}$ is said to be \ti{compactly generated} if there is a set $\G$ of compact objects of $\msf{T}$ such that $X \simeq 0$ in $\T$ if and only if $H_i^G(X) = 0$ for all $G \in \G$ and $i \in \mbb{Z}$. Recall that, if $\msf{X}\subseteq\T$ is a class of objects in $\T$, then $\t{Add}(\msf{X})$ denotes the class of objects obtained as retracts of arbitrary coproducts of objects in $\msf{X}$. Similarly, $\t{Prod}(\msf{X})$ is the class of objects obtained as retracts of arbitrary products of $\msf{X}$. We denote by $\t{thick}(\msf{X})$ the smallest triangulated subcategory of $\T$ closed under retracts and containing $\msf{X}$, while $\t{Loc}(\msf{X})$ denotes the smallest thick subcategory of $\T$ containing $\msf{X}$ which is closed under arbitrary coproducts.

From now on, we will always assume that $\msf{T}$ is compactly generated, and let $\msf{T}^{\c}$ denote the full subcategory of $\msf{T}$ consisting of the compact objects. Recall that $\Mod{\T^\c}$ denotes the category of additive functors $(\T^\c)^\t{op} \to \ab$. The restricted Yoneda embedding $\tb{y}\colon \msf{T}\to \Mod{\T^\c}$ is defined by
\[
\tb{y}X=\Hom_{\msf{T}}(-,X)\vert_{\T^{\c}}
\]
on objects and with the obvious action on morphisms. This functor enables a definition of purity, analogous to the one given in \cref{puredef}.

\begin{defn}[{\cite[1.1]{krsmash}}]\label{tpuredef}
A triangle 
\[
\begin{tikzcd}
X\arrow{r}{\alpha} & Y \arrow{r}{\beta} & Z \arrow{r}{f} & \Sigma X
\end{tikzcd}
\]
is said to be \ti{pure} if the induced sequence $0\to \tb{y}X\to\tb{y}Y\to\tb{y}Z\to 0$ is exact in $\Mod{\T^\c}$. In this case, we say that $\alpha$ is a \ti{pure monomorphism} and $\beta$ is a \ti{pure epimorphism}, and call $X$ a pure subobject, and $Z$ a pure quotient, of $Y$.
\end{defn}

We say that a class $\msf{X}\subseteq\T$ is closed under pure subobjects if for any pure monomorphism $A\to X$ in $\T$ with $X\in\msf{X}$ we have $A\in\msf{X}$. We dually define what it means to be closed under pure quotients. We say $\msf{X}$ is closed under pure extensions if given a pure triangle $X\to Y\to Z\to \Sigma X$ with $X,Z\in\msf{X}$, we have $Y\in\msf{X}$.

Note that this formulation of purity is nothing other than \cref{puredef}(1) reformulated for $\T$. In fact, there are equivalent formulations of purity in $\T$ that further highlight the similarities between \cref{tpuredef} and \cref{puredef}. A direct analogy of \cref{puredef}(3) can be found at \cite[2.8]{krcoh}. The aim of this section is to prove an equivalent definition of purity in $\msf{T}$ that is the analogue of \cref{puredef}(4). 

In the previous definition, the induced morphism $\tb{y}f$ being the zero map is equivalent to having $\Hom_{\T}(C,f)=0$ for all $C\in\msf{T}^\c$. Such maps have already been well studied in the literature.
\begin{defn}
A morphism $f\in\T$ is \ti{phantom} if $\Hom_{\T}(C,f)=0$ for all $C\in\T^{\c}$.
\end{defn}

Another crucial concept that will regularly be used throughout is that of pure injective objects. They are, as the following shows, closely related to phantom maps.

\begin{defn}[{\cite[1.8]{krsmash}}]
An object $X\in\T$ is \ti{pure injective} if it satisfies any of the following equivalent conditions:
\begin{enumerate}
\item $\tb{y}{X}$ is injective in $\Mod{\T^\c}$;
\item any pure monomorphism $X \to Y$ splits;
\item for every set $I$, the summation $\oplus_I X \to X$ factors through the canonical map $\oplus_I X \to \prod_I X$;
\item any phantom map $Y \to X$ is zero.
\end{enumerate}
\end{defn}

In order to give an analogue of \cref{puredef}(4), we introduce the notion of a duality triple.
\begin{defn}\label{dualitytriple}
A \ti{duality triple} $(\T,\msf{U},Q)$ consists of two compactly generated triangulated categories $\T$ and $\msf{U}$, together with exact functors $Q\colon\T^{\t{op}}\to\msf{U}$ and $Q\colon\msf{U}^{\t{op}}\to\T$ such that:
\begin{enumerate}
\item $Q(X)$ is pure injective for any $X\in\T,\msf{U}$;
\item for any $X\in\T,\msf{U}$, there is a natural map $i_{X}\colon X\to Q^{2}(X)$ which is a pure monomorphism;
\item the composite $Q(i_X) \circ i_{Q(X)}$ is the identity for any $X \in \T, \msf{U}$;
\item for any set of objects $\{X_{i}\}_{I}$ in $\T$ or $\msf{U}$, we have $Q(\oplus_{I}X_{i})\simeq \prod_{I}Q(X_{i})$.
\end{enumerate}
\end{defn}

\begin{rem}
Notice that the order of the triangulated categories in the triple is irrelevant. We also abuse notation in the definition of duality triple by using $Q$ to denote both functors, even though they need not be the same. This is justified by the symmetry in the definition, and by the examples given in the later sections (see \cref{ttdualitytriple} and \cref{algebraicdualitytriple}). One can immediately see the relationship between $Q$ and $\Hom_{\Z}(-,\qz)$ in the module case.
\end{rem}

Our first main result provides a triangulated analogue of \cref{puredef}(4), and will be of significant use in the subsequent sections.

\begin{thm}\label{detects}
Let $(\T,\msf{U},Q)$ be a duality triple. Then a morphism $f\colon X\to Y$ in $\T$ is phantom if and only if $Q(f)$ is zero. In particular, $A\to B\to C$ is a pure triangle in $\T$ if and only if the induced triangle $Q(C)\to Q(B)\to Q(A)$ splits. 
\end{thm}

The proof of the theorem relies on the following lemma, which gives an alternative characterisation of phantom maps.

\begin{lem}\label{phantomcomp}
Let $(\T,\msf{U},Q)$ be a duality triple and $f\colon X\to Y$ in $\T$. Then $f$ is phantom if and only if the composite $i_{Y}\circ f\colon X\to Y\to Q^{2}Y$ is zero.
\end{lem}
\begin{proof}
If $f$ is phantom then the composite $i_Y \circ f$ is also phantom, since if $C$ is compact then $C\to X\to Y$ is zero by definition. Then since $i_{Y}\circ f$ is a phantom map into a pure injective object, it must be zero by \cite[1.8]{krsmash}. Before showing the reverse implication, note that \cite[4.12]{phantom} holds in our setting - this is because $\msf{U}$ is compactly generated and the category of additive functors $\T^{\c}\to\tb{Ab}$ has enough injectives. Thus, as in \cite[4.12]{phantom}, $f$ is phantom if and only if the composite $X\to Y\to Q(C)$ is zero for every $C\in\msf{U}^{\c}$ and each map $Y\to Q(C)$. Now, by the definition of a duality triple, the map $i_{Y}\colon Y \to Q^2(Y)$ is a pure monomorphism and $Q(C)$ is pure injective. Therefore the map $Y\to Q(C)$ factors as $Y\to Q^{2}(Y)\to Q(C)$ by \cite[1.4]{krsmash}. It therefore follows that $X\to Y\to Q(C)$ is zero, hence $f$ is phantom.
\end{proof}  

\begin{proof}[Proof of \cref{detects}]
By condition (3) of \cref{dualitytriple}, we have $Q(f) = Q(f) \circ (Q(i_{Y}) \circ i_{Q(Y)}) = Q(i_{Y}\circ f)\circ i_{Q(Y)}$. In particular, if $f$ is phantom then we see that $Q(f)=0$ since \cref{phantomcomp} tells us that $i_Y \circ f = 0$. For the converse, by the naturality of $i$ we have a commutative diagram
\[
\begin{tikzcd}[column sep = 0.75in, row sep = 0.6in]
	X \arrow{r}{f} \arrow[d, "i_{X}"'] & Y \arrow{d}{i_{Y}} \\
	Q^{2}(X) \arrow[r, "Q^{2}(f)"'] & Q^{2}(Y)
\end{tikzcd}\]
in $\T$. If $Q(f)=0$, then $Q^{2}(f)=0$, and consequently by the above commutative diagram we see the composition $i_{Y}\circ f$ factors through the zero map, and hence by \cref{phantomcomp}, $f$ is phantom. 

Let us now prove the statement about pure triangles. If $A\to B\to C\xrightarrow{f}\Sigma A$ is a pure triangle, then the map $f$ is phantom, and therefore vanishes under $Q$. In particular, the zero map appears in the triangle
\[
\Sigma Q(A)\xrightarrow{0} Q(C)\to Q(B)\to Q(A)
\]
and therefore the triangle is split. Conversely, if the induced triangle splits then $Q(f) = 0$ (see for example~\cite[1.4]{Happel}). Therefore $f$ is phantom and hence the triangle $A \to B \to C \to \Sigma A$ is pure.
\end{proof}

\begin{rem}
	Following the completion of this manuscript, we became aware that in the case when $\T=\D(R)$ for $R$ a ring and $Q$ the derived character dual, the above result was independently obtained in~\cite[2.6]{HrbekHugel}. As we show in ~\cref{algebraicdualitytriple} (and~\cref{ttdualitytriple}), the abstract framework of duality triples incorporates this result as a particular case of a much more general phenomenon.
\end{rem}

We now give two examples of classes of compactly generated triangulated categories which naturally admit duality triples, for which, therefore, the above theorem applies.

\subsection{Duality for big tensor-triangulated categories}
Recall that a \ti{tensor-triangulated category} is a triangulated category with a closed symmetric monoidal structure that is compatible with the shift $\Sigma$ on $\T$; see \cite[\S A.2]{axiomatic} for a detailed definition. We will let $\otimes$ denote the monoidal product on $\T$, with unit $\mathbbm{1}$, and $F(-,-)$ will denote the internal hom (the right adjoint to $\otimes$). An object $X \in \T$ is \emph{rigid} if the natural map
\[
DX\otimes Y \to F(X,Y)
\] 
is an isomorphism for every $Y\in \T$, where $DX=F(X,\mathbbm{1})$ denotes the functional dual.

\begin{defn}
A \ti{big tensor-triangulated} category is a compactly generated tensor-triangulated category $\T$ for which the rigid and compact objects coincide. 
\end{defn}

Such categories are sometimes called rigidly-compactly generated triangulated categories, but for brevity we use `big'. We list some common examples of such categories; more information can be found in \cite{axiomatic}. Observe that the tensor unit $\1$ is compact in a big tensor-triangulated category since it is rigid.

\begin{ex}\leavevmode
	\begin{enumerate}
		\item If $R$ is a commutative ring, then $\D(R)$, the derived category of $R$-modules is a big tensor-triangulated category. The ring itself is a compact generator.
		
		\item If $X$ is a quasicompact and quasiseparated scheme, then $\D_{\t{qc}}(X)$, the unbounded derived category of complexes of sheaves of $\mc{O}_X$-modules whose cohomology modules are quasi-coherent, is a big tensor-triangulated category.
		
		\item For a finite group $G$ and a field $k$ whose characteristic divides the order of $G$, the stable module category $\t{StMod}(kG)$ is a big tensor-triangulated category, with a compact generating set given by the set of simple $kG$-modules.
		
		\item The stable homotopy category of spectra, $\t{Sp}$, is a big tensor-triangulated category, generated by the sphere $S^{0}$. More generally, the derived category of a commutative ring spectrum is a big tensor-triangulated category.
		
		\item Given a compact Lie group $G$, the equivariant stable homotopy category, $\t{Sp}_{G}$, is a big tensor-triangulated category. The set $\{G/H_{+}\}$, as $H$ ranges over closed subgroups of $G$, gives a set of compact generators. 
	\end{enumerate}
\end{ex}

Given a big tensor-triangulated category $\T$, we have a naturally occurring functor $\T^{\t{op}}\to \T$ that mimics $\Hom_{\Z}(-,\qz)$ as in \cref{puredef} which can be constructed using Brown representability. Recall that Brown representability states that any cohomological functor $F\colon\T^{\t{op}}\to \tb{Ab}$ that sends coproducts in $\T$ to products is of the form $\Hom_{\T}(-,T)$ for some $T\in \T$ (see \cite{Neemantri}).

\begin{defn}\label{browncomenetz}
Let $\T$ be a big tensor-triangulated category, and $C\in\T^{\c}$ and $X\in\T$. The \ti{Brown--Comenetz dual} of $X$ with respect to $C$ is the unique object, denoted $\mbb{I}_{C}(X)$, which represents the cohomological functor
\[
\Hom_{\Z}(H_{0}^{C}(X\otimes -),\qz),
\]
where $H_{0}^{C}(-)=\Hom_{\T}(C,-)$.
\end{defn}
As we see in the next lemma, which collates some useful consequences of the definition and properties of Brown--Comenetz duals, it is often enough to use the Brown--Comenetz dual of $\mathbbm{1}$ with respect to itself. As is standard, we write $\bc{X} := \bctwo{\1}{X}$ for the Brown--Comenetz dual of $X$ with respect to the unit, and $\mbb{I} := \bc{\1}$ for the Brown--Comenetz dual of the tensor unit.

\begin{lem}\label{lem:BCproperties} Let $\T$ be a big tensor-triangulated category and let $X, Y \in \T$ and $C \in \T^{\c}$.
	\begin{itemize}
		\item[(i)] We have a natural isomorphism $\bctwo{C}{X} = \bc{(F(C, X))}$.
		\item[(ii)] There are natural isomorphisms $\bc{X} = F(X, \mbb{I})$ and $\bctwo{C}{X} = F(F(C,X), \mbb{I}).$
		\item[(iii)] We have $\Hom_\T(\Sigma^nY, \bctwo{C}{X}) = \Hom_\Z(H_{-n}^C(X \otimes Y), \Q/\Z).$
		\item[(iv)] The Brown--Comenetz dual functor $\mbb{I}$ is conservative, i.e., for any object $Z \in \T$ we have $Z \simeq 0$ if and only if $F(Z, \mbb{I}) \simeq 0$.
	\end{itemize}
\end{lem}

\begin{proof}
For the first item, note that there are isomorphisms
\begin{align*}
 \Hom_\T(Y, \bctwo{C}{X}) &= \Hom_\Z(H^C_0(X \otimes Y), \Q/\Z) \\
 &= \Hom_\Z(H^\1_0(F(C,X) \otimes Y), \Q/\Z) \mbox{ as $C$ is rigid}\\
 &= \Hom_\T(Y, \bc{(F(C,X))}) \mbox{ by adjunction}
\end{align*}
which proves part (i). The first part of (ii) is also clear by adjunction, while the second part follows by applying (i). Part (iii) follows from the defining property of the Brown--Comenetz dual. For part (iv), assume that $F(Z,\mbb{I}) \simeq 0$. Therefore $F(Z, \mbb{I}_C) \simeq F(Z, \mbb{I}) \otimes C \simeq 0$ for all $C \in \T^\c$ using part (ii). Applying $\Hom_\T(\Sigma^i\1,-)$ we obtain that $\Hom_\T(\Sigma^iZ, \mbb{I}_C) = \Hom_\Z(H^C_{-i}Z, \Q/\Z) = 0$ for each $i \in \mathbb{Z}$ and $C \in \T^\c$, and since $\Q/\Z$ is a cogenerator for $\Mod{\Z}$ it follows that $H^C_*Z = 0$ for each $C \in \T^\c$, and hence $Z \simeq 0$.
\end{proof}

The following proposition shows that every big tensor-triangulated category naturally admits the structure of a duality triple.
\begin{prop}\label{ttdualitytriple}
	For any big tensor-triangulated category $\T$, the triple $(\T,\T,\mbb{I})$ is a duality triple. 
\end{prop}
\begin{proof}
Conditions (1) and (2) of \cref{dualitytriple} follow from work of Christensen-Strickland~\cite[3.12, 4.13]{phantom}, and conditions (3) and (4) are clear by construction.
\end{proof}

\subsection{Duality for compactly generated algebraic triangulated categories}
A triangulated category $\T$ is \ti{algebraic} if it is equivalent as a triangulated category to the stable category of a Frobenius exact category. As shown in \cite{keller}, where the notion was first introduced, any compactly generated algebraic triangulated category is triangulated equivalent to the derived category of modules over a small differential graded category. We will use this to establish a duality triple for compactly generated algebraic triangulated categories.

Recall that a \ti{differential graded category} (henceforth, dg-category) is a category $\mc{A}$ that is enriched over $\msf{Ch}(R)$, the category of chain complexes over a commutative ring $R$. More explicitly, $\mc{A}$ is a category and for any pair of objects $X,Y\in\mc{A}$, $\Hom_{\mc{A}}(X,Y)$ is a chain complex over $R$, and the composition map is linear.

A \ti{dg-functor} is a functor $F\colon\mc{A}\to\mc{B}$ between dg-categories (implicitly over $R$) such that the natural map
\[
\Hom_{\mc{A}}(X,Y)\to \Hom_{\mc{B}}(FX,FY)
\]
is a chain map in $\msf{Ch}(R)$. For a dg-category $\mc{A}$, a left \ti{dg-$\mc{A}$-module} is a dg-functor $\mc{A}\to\msf{Ch}(R)$, where we view $\msf{Ch}(R)$ as a dg-category by letting $\Hom_{\msf{Ch}(R)}(X,Y)$ denote the Hom-complex. A right dg-$\mc{A}$-module is a dg-functor $\mc{A}^{\t{op}}\to\msf{Ch}(R)$. We will let $\Mod{\mc{A}}$ and $\Mod{\mc{A^{\circ}}}$ denote the categories of left and right $\mc{A}$-modules respectively. The morphisms in these categories are natural transformations of dg-functors, which themselves take the structure of a chain complex over $R$, hence $\Mod{\mc{A}}$ and $\Mod{\mc{A}^{\circ}}$ are differential graded categories themselves. More details can be found in \cite{keller2}.

For a differential graded category $\mc{A}$, define $\msf{C}(\mc{A})$ to be the category whose objects are the same as those of $\Mod{\mc{A}}$, but whose morphisms are given by
\[
\msf{C}(\mc{A})(X,Y)=Z^{0}\Hom_{\Mod{\mc{A}}}(X,Y).
\]
Essentially, if $X$ and $Y$ are $\mc{A}$-modules, then one can view the morphisms of $\msf{C}(\mc{A})$ as the chain maps between $X(A)$ and $Y(A)$ for any object $A\in\mc{A}$. It is shown in \cite[2.12]{sp}, that if $\mc{A}$ is a small differential graded category then $\msf{C}(\mc{A})$ is a locally finitely presented Grothendieck category, as it is equivalent to $\Mod{\mc{R}}$ for some small preadditive category $\mc{R}$. The same works for right $R$-modules, and $\msf{C}(\mc{A}^{\circ})\simeq \Mod{\mc{R}^{\circ}}$.
The character dual $(-)^+\colon \Mod{\mc{R}}\to \Mod{\mc{R}^{\circ}}$ defined by $M^{+} = \Hom_{\Z}(M(-),\qz)$, has the property that the image of any module under it is pure injective, by \cite{stenstrom}. As such we obtain a character dual $(-)^+\colon \msf{C}(\mc{A}) \to \msf{C}(\mc{A}^\circ)$.
 
The \ti{derived category} of $\mc{A}$, denoted $\D(\mc{A})$, is the localisation of $\msf{C}(\mc{A})$ along all quasi-isomorphisms, and we write $\msf{q}\colon\msf{C}(\mc{A})\to\D(\mc{A})$ for the associated localisation functor. The functor $(-)^+\colon \msf{C}(\mc{A}) \to \msf{C}(\mc{A}^\circ)$ has an induced derived functor $\D(\mc{A}) \to \D(\mc{A}^\circ)$ which by abuse of notation we also denote by $(-)^+$.

\begin{prop}\label{algebraicdualitytriple}
Given any compactly generated algebraic triangulated category $\T\simeq \D(\mc{A})$, where $\mc{A}$ is a small dg-category, we have that $(\D(\mc{A}), \D(\mc{A}^\circ), (-)^+)$ is a duality triple.
\end{prop}
\begin{proof}
In any finitely accessible category with products and coproducts, an object $X$ is pure injective if and only if the canonical summation $\oplus_I X\to X$ factors through the map $\oplus_I X\to \prod_I X$ for any set $I$~\cite[5.4]{dac}. The same is true in any compactly generated triangulated category by~\cite[1.8]{krsmash}. Therefore as $\msf{q}$ commutes with both coproducts and products, it preserves pure injectives, and hence, for any object $X\in\D(\mc{A})$, the object $X^{+}$ is pure injective in $\D(\mc{A}^{\circ})$. 

Let $X \in \D(\mc{A})$, and note that we may view $X$ as an object in $\msf{C}(\mc{A})$. There is a morphism $j_X\colon X\to X^{++}$ in $\msf{C}(\mc{A})$ so there is a canonical morphism $i_X := \msf{q}(j_X)\colon X\to X^{++}$ in $\D(\mc{A})$. To show this is a pure monomorphism, we adapt the proof of \cite[3.8]{lv}. The map $j_X$ is a pure embedding as $\msf{C}(\mc{A})$ is a module category. In particular, there is a directed system $\{g_{i}\}_{I}$ of split monomorphisms such that $\rlim_{I}g_i=j_X$. Yet \cite[3.5]{lv} tells us that the composition 
\[
\msf{C}(\mc{A})\xrightarrow{\msf{q}} \D(\mc{A}) \xrightarrow{\mathbf{y}} \Mod{\D(\mc{A})^\c}
\]
preserves directed colimits, where $\mathbf{y}$ is the restricted Yoneda embedding. In particular, $(\mathbf{y}\circ \msf{q})(j_X) = \rlim_{I}(\mathbf{y}\circ \msf{q})(g_{i})$. 
Now, each $\msf{q}(g_{i})$ is a split monomorphism, and therefore the homology colimit $\mrm{hocolim}_{I} \msf{q}(g_i)$ of this directed system of split monomorphisms is a pure monomorphism by~\cite[2.8]{krcoh}. By definition, \[\mathbf{y}\mrm{hocolim}_{I} \msf{q}(g_i) = \rlim_{I}(\mathbf{y}\circ \msf{q})(g_{i})= (\mathbf{y}\circ \msf{q})(j_X) = \mathbf{y}i_X.\] Therefore $\mathbf{y}i_X$ is a monomorphism, and hence $i_X$ is a pure monomorphism as required. 

The third condition required is, again, one that follows from purity in module categories: the object $X^{+}$ is pure injective, and therefore the canonical embedding $j_{X^{+}}\colon X^{+}\to X^{+++}$ splits, with splitting map given by $(j_{X})^{+}$. As $\msf{q}$ preserves splittings, and $i_X = \msf{q}(j_X)$, the third condition follows. The fourth condition is clear by definition.
\end{proof}

We end this section by giving some examples of compactly generated algebraic triangulated categories to which one may apply the previous result.

\begin{ex}\leavevmode
\begin{enumerate}
	\item If $\mc{A}$ is a small dg-category, then $\D(\mc{A})$ is a compactly generated algebraic triangulated category; indeed, every example takes this form up to triangulated equivalence. In particular, so is $\D(R)$ for any dga $R$, by considering $R$ as a differential graded category with one object. When $R$ is a ring, one may easily check that the duality functor on $\msf{D}(R)$ constructed above in \cref{algebraicdualitytriple} is indeed the derived character dual $\msf{R}\Hom_\Z(-, \qz)$.
	
	\item For a finite group $G$ and a field $k$ whose characteristic divides the order of $G$, the stable module category $\t{StMod}(kG)$ is a compactly generated algebraic triangulated category. 
	
	\item If $\mc{B}$ is an additive category, then the homotopy category $\msf{K}(\mc{B})$ of $\mc{B}$ is an algebraic triangulated category. However, in general it is not compactly generated; for example, $\msf{K}(R)$ (for a ring $R$) is typically not compactly generated. Nonetheless many examples are known for which it is. Generalizing a result of J\o rgensen~\cite{Jproj}, Neeman~\cite{Nproj} showed that the homotopy category $\msf{K}(\t{Proj}(R))$ of all projectives, is compactly generated whenever $R$ is right coherent. The methods used were then utilised in \cite{Sproj}, to show that $\msf{K}(\t{Proj}(\mc{B}))$ is compactly generated whenever $\mc{B}$ is a skeletally small additive category with weak cokernels. For injective objects, it was shown in \cite{Krausest} that $\msf{K}(\t{Inj}(\mc{A}))$ is compactly generated whenever $\mc{A}$ is a locally noetherian Grothendieck category. If one only wishes to study classes of modules, a sufficient condition for compact generation of $\msf{K}(\mc{X})$, where $\mc{X}\subseteq\Mod{R}$ can be found in \cite{hjcg}.
\end{enumerate}
	
\end{ex}

\section{Duality pairs and their closure properties}
In this section, we introduce the triangulated formulation of duality pairs, directly following Holm and J{\o}rgensen, before highlighting their role in purity and approximation theory. 

\begin{defn}\label{tdp}
Let $(\T, \msf{U}, Q)$ be a duality triple. A pair of classes $(\A,\B)$, where $\A\subseteq\T$ and $\B\subseteq\msf{U}$, is a \ti{duality pair} if:
\begin{enumerate}
	\item $X\in\A$ if and only if $Q(X)\in\B$;
	\item $\B$ is closed under finite coproducts and retracts.
\end{enumerate}
\end{defn}

Note that it is immediate that the class $\A$ is also closed under finite coproducts and retracts, as $Q$ is an additive functor. In the subsequent section we provide several examples of duality pairs, and see that they arise when considering natural questions relating to triangulated categories. 

\subsection{Closure properties} Let us now show some closure properties of duality pairs that can be deduced directly from the definition.

\begin{prop}\label{closure}
Let $(\A,\B)$ be a duality pair on a duality triple $(\T, \msf{U}, Q)$.
\begin{enumerate}
	\item If $\B$ is closed under arbitrary products, then $\A$ is closed under arbitrary coproducts.
	\item If $\B$ is triangulated, then so is $\A$.
\end{enumerate}
Moreover, if $(\A,\B)$ is a duality pair on the duality triple $(\T,\T,\mbb{I})$ for a big tensor-triangulated category $\T$ then we additionally have:
\begin{enumerate}[resume]
	\item If $\B$ is $F$-closed (that is, for any $B\in\B$ and $X\in \T$ we have $F(X,B)\in\B$), then $\A$ is a $\otimes$-ideal (that is, for any $A \in \A$ and $X \in \T$ we have $X \otimes A \in \A$).
\end{enumerate}
\end{prop}

\begin{proof}
For the first item, suppose $\B$ is closed under products and $\{X_{i}\}_{I}$ is a collection of objects in $\A$. Then $Q(\oplus_{I}X_{i})\simeq \prod_{I}Q(X_{i})$. As $(\A,\B)$ is a duality pair, each $Q(X_{i})$ is in $\B$, and hence, by assumption, so is $\prod_{I}Q(X_{i})$. But then, by definition, we have $\oplus_{I}X_{i}\in\A$. For the second, suppose that $X\to Y\to Z$ is a triangle with $X$ and $Z$ in $\A$. Then $Q(Z)\to Q(Y)\to Q(X)$ is a triangle with $Q(X)$ and $Q(Z)$ in $\B$. In particular, if $\B$ is closed under triangles, we immediately see that $Q(Y)\in\B$, hence $Y\in\A$ by definition. As $\B$ is closed under shifts, we can similarly argue that $\A$ is. Lastly, suppose that $\T$ is tensor-triangulated and $A\in\A$. If $X\in\T$ is arbitrary, then $A\otimes X\in\A$ if and only if $F(A\otimes X,\mbb{I})\in\B$. By adjunction, we have $F(A\otimes X,\mbb{I})\simeq F(X,F(A,\mbb{I}))$; yet $F(A,\mbb{I})\in\B$ and $\B$ is $F$-closed, which proves the claim.
\end{proof}

We now highlight the connection between duality pairs, pure closure properties, and approximations, providing a direct analogy of \cref{hjtheorem}.

\begin{thm}\label{dp}
Let $(\A,\B)$ be a duality pair on a duality triple $(\T, \msf{U}, Q)$. Then $\A$ is closed under pure subobjects, pure quotients, and pure extensions. Moreover, if $\T$ is algebraic, then:
\begin{enumerate}
	\item if $\A$ is closed under coproducts it is precovering;
	\item if $\A$ is closed under products it is preenveloping;
	\item if $\A$ is closed under coproducts and triangles, it is a torsion class.
\end{enumerate}
\end{thm}

\begin{proof}
The proof is similar to that of \cref{hjtheorem}. Suppose that $X\to Y\to Z$ is a pure triangle in $\T$. Then, applying $Q$, we obtain a triangle 
\[
Q(Z)\to Q(Y)\to Q(X)
\]
in $\msf{U}$ which is split by \cref{detects}, so $Q(Y)\simeq Q(X)\oplus Q(Z)$. In particular, we immediately deduce the claims about pure subobjects, quotients and extensions. For the enumerated claims, we apply the closure properties to \cite[4.2, 5.2]{lv}.
\end{proof}

\begin{rem}
We note that in the case that $\A$ is triangulated in $\T$, it is precovering if and only if it is closed under coproducts~\cite[1.4]{adj}.
\end{rem}

The following, rather immediate, lemma will be of use in several later applications.
\begin{lem}\label{coprodclosed}
	Let $\A$ be a class of objects closed under pure subobjects. If $\A$ is closed under products then it is closed under coproducts. In particular, if $(\A,\B)$ is a duality pair and $\A$ is closed under products, then it is also closed under coproducts.
\end{lem}
\begin{proof}
	In any compactly generated triangulated category, the canonical map $\oplus_{I}X_{i}\to\prod_{I}X_{i}$ is a pure monomorphism which gives the first claim. For the second claim, note that $\A$ is closed under pure subobjects by \cref{dp}.
\end{proof}

We have already seen that duality pairs have good closure properties. However, if we impose another condition on the duality pairs we consider, one can deduce stronger closure properties as we now show. Moreover, examples of such duality pairs still abound.

\begin{defn}
A duality pair $(\A,\B)$ is \ti{symmetric} if both $(\A,\B)$ and $(\B,\A)$ are duality pairs.
\end{defn}

\begin{prop}\label{symmclosure}
If $(\A,\B)$ is a symmetric duality pair, then the following hold:
\begin{enumerate}
	\item Both $\A$ and $\B$ are closed under pure injective envelopes.
	\item $\A$ is triangulated if and only if $\B$ is triangulated.
	\item $X\in \A$ if and only if $Q^{2}X\in\A$, and likewise for $\B$.
	\item The following statements are equivalent:
	\begin{enumerate}
		\item $\A$ is closed under coproducts;
		\item $\A$ is closed under products;
		\item $\B$ is closed under coproducts;
		\item $\B$ is closed under products.
	\end{enumerate}
\end{enumerate}
\end{prop}

\begin{proof}
For the first item, suppose that $A\in\A$. Then as $(\A,\B)$ is symmetric, it is clear that $Q^{2}A\in\A$. As described in \cite{krsmash}, pure injective objects and envelopes correspond with injective objects and envelopes in the functor category $\Mod{\T^\c}$. By considering the restricted Yoneda embedding of all the maps, it follows that the pure injective hull of $A$ is a summand of $Q^{2}(A)$. Yet we know $\A$ is closed under summands, proving the claim.
The second item follows from \cref{closure}. The third claim follows from \cref{dp} since $X\to Q^{2}X$ is a pure monomorphism. 

Let us now prove the fourth claim. The implications $(b)\implies (a)$ and $(d)\implies (c)$ follow from \cref{coprodclosed}, while $(b)\implies (c)$ and $(d)\implies (a)$ hold by \cref{closure}. We show that $(c)\implies (b)$; then $(a)\implies (d)$ follows via symmetry. Suppose that $\{X_{i}\}_{I}$ is a collection of objects in $\A$, so $QX_{i}\in\B$ for each $i\in I$. Since $\B$ is closed under coproducts by assumption we have $\oplus_{I}QX_{i}\in\B$. As $(\B,\A)$ is also a duality pair, we have $Q(\oplus_{I}QX_{i})\simeq \prod_{I}Q^{2}X_{i}\in\A$. Now, for each $i$, the morphism $X_i\to Q^{2}X_i$ is a pure monomorphism, and products of pure monomorphisms are also pure, hence
\[
\prod_{I}X_{i}\to\prod_{I}Q^{2}X_{i}
\]
is a pure monomorphism. Consequently $\prod_{I}X_{i}\in \A$ by \cref{dp}.
\end{proof}

\subsection{The minimal duality pair generated by a class}\label{minimaldualitypair}
We now show that not only do duality pairs on a given duality triple always exist, but every class of objects generates a minimal one. We will use this construction later in \cref{dplift} to relate duality pairs on triangulated categories to those in functor categories. Let us fix a duality triple $(\T, \msf{U}, Q)$. For any class of objects $\mc{S}\subseteq\T$, define $Q\mc{S}=\{QX:X\in\mc{S}\}$. Using standard notation, we let $\t{add}(Q\mc{S})$ denote the smallest subcategory of $\msf{U}$ containing $Q\mc{S}$ which is closed under retracts and finite coproducts.
We then define
\[
\overline{\mc{S}}=\{X\in\T:QX\in\t{add}(Q\mc{S})\} \subseteq \T.
\]

\begin{lem}\label{gendp}
	Let $(\T, \msf{U}, Q)$ be a duality triple and let $\mc{S} \subseteq \T$. The pair $(\overline{\mc{S}},\mrm{add}(Q\mc{S}))$ is a duality pair.
\end{lem}

\begin{proof}
	It is clear that $X\in\overline{\mc{S}}$ if and only if $QX\in\t{add}(Q\mc{S})$ via the way the classes are defined. By definition, $\t{add}(Q\mc{S})$ is closed under finite coproducts and retracts.
\end{proof}

\begin{defn}
	Let $(\T, \msf{U}, Q)$ be a duality triple. Given a class $\mc{S}\subseteq\mathsf{T}$, we say that $(\overline{\mc{S}},\t{add}(Q\mc{S}))$ is the \ti{duality pair generated by} $\mc{S}$, and $\overline{\mc{S}}$ is the \ti{duality class} generated by $\mc{S}$.
\end{defn}

We now show that using the terminology `minimal' and `generated' is justified.

\begin{lem}\label{barclosure}
	Let $(\T, \msf{U}, Q)$ be a duality triple and let $\mc{S} \subseteq \T$. Suppose $(\mathsf{A},\mathsf{B})$ is a duality pair with $\mc{S}\subseteq \mathsf{A}$, then $\overline{\mc{S}}\subseteq \mathsf{A}$.
\end{lem}

\begin{proof}
	Let $X\in\overline{\mc{S}}$, so $QX$ is a summand of $Q(S_{1})\oplus\cdots\oplus Q(S_{n})$ for some $S_{i}\in\mc{S}$, $1\leq i\leq n$. But as $\mc{S}\subseteq\mathsf{A}$, we have $Q(S_{i})\in\mathsf{B}$, and as $\B$ is closed under finite coproducts, we have that $QX$ is a summand of an object of $\mathsf{B}$. As $\B$ is closed under summands, $QX$ is itself an element of $\mathsf{B}$. But by the definition of duality pairs, it follows that $X\in\mathsf{A}$.
\end{proof}

\begin{cor}
	Let $(\mathsf{A},\mathsf{B})$ be a duality pair on a duality triple $(\T, \msf{U}, Q)$. Then $\overline{\mathsf{A}}=\mathsf{A}$.
\end{cor}
\begin{proof}
	Clearly $\mathsf{A}\subseteq \overline{\mathsf{A}}$. On the other hand, by \cref{barclosure} we have $\overline{\mathsf{A}}\subseteq\mathsf{A}$.
\end{proof}

Any duality class is closed under finite coproducts, pure extensions, pure subobjects and pure quotients by \cref{dp}. However, it is not clear that these properties uniquely determine duality classes. The following illustrates a sufficient (but very far from necessary) condition to determine that a class is a duality class.

\begin{lem}
	Let $(\T, \msf{U}, Q)$ be a duality triple and let $\mc{S} \subseteq \T$. Suppose $\mc{S}$ is closed under pure submodules and finite coproducts and that $Q^{2}\mc{S}\subseteq\mc{S}$. Then $\mc{S}$ is a duality class.
\end{lem}

\begin{proof}
	Since $\mc{S} \subseteq \overline{\mc{S}}$, it suffices to show that $\overline{\mc{S}}\subseteq\mc{S}$. Let $X\in\overline{\mc{S}}$ so $QX$ is a summand of $QS$ for some $S\in\mc{S}$. Then there is a composition of pure embeddings
	\[
	X\to Q^{2}X\to Q^{2}S
	\]
	with the latter object being in $\mc{S}$ by assumption. But since the composition of pure embeddings is pure, it follows, from the assumption that $\mc{S}$ is closed under pure submodules, that $X\in\mc{S}$.
\end{proof}

\section{Examples of duality pairs}
In this section, we provide some examples of duality pairs arising in algebra and topology. 

\begin{ex}
Let $R$ be a ring and consider the derived category $\D(R)$ of left $R$-modules. Recall that the \ti{flat dimension} of $X\in\D(R)$ is given by 
\[
\t{fd}\,X=\sup\{n \in \mathbb{Z} : H_n(R/\mf{a}\otimes_{R}^{\mathsf{L}}X) \neq 0 \t{ where $\mf{a}$ is a finitely generated right ideal of $R$}\},
\]
while the \ti{injective dimension} of $Y\in\D(R^{\circ})$ is given by
\[
\t{id}\,Y= \sup\{n \in \mathbb{Z} : H_{-n}\msf{R}\Hom_{R}(R/\mf{a},Y) \neq 0 \t{ where $\mf{a}$ is a finitely generated right ideal of $R$}\}.
\]
For any $n\geq 0$, let $\mc{F}_{n}\subset\D(R)$ denote the class of complexes of flat dimension at most $n$, and similarly define $\mc{I}_{n}\subset\D(R^{\circ})$ to be the class of complexes of injective dimension at most $n$. Then $(\mc{F}_{n},\mc{I}_{n})$ is a duality pair on the duality triple $(\D(R), \D(R^\circ), \msf{R}\Hom_\Z(-, \qz))$. It is clear that $\mc{I}_{n}$ is closed under retracts and finite coproducts, while the equalities
\[
H_{-n}\msf{R}\Hom_{R}(R/\mf{a},\msf{R}\Hom_{\Z}(X,\qz))= H_{-n}\msf{R}\Hom_{\Z}(R/\mf{a}\otimes_{R}^{\mathsf{L}}X,\qz) = H_n(R/\mf{a}\otimes_{R}^{\mathsf{L}}X),
\]
which hold for any $X\in\D(R)$ by \cite[2.5.7]{dcmca}, show that $Y\in\mc{F}_{n}$ if and only if $\msf{R}\Hom_{\Z}(Y,\qz)\in\mc{I}_{n}$. Note that this example is classic and can be seen immediately from \cite[4.1]{af}. The same reasoning also shows that if $R$ is right coherent, the complexes $\mc{GF}_{n}$ (resp., $\mc{GI}_{n}$) of Gorenstein flat (resp., Gorenstein injective) dimension at most $n$ also appear in the analogous duality pair $(\mc{GF}_{n},\mc{GI}_{n})$, using \cite[3.11]{holmgor}. 
\end{ex}

\begin{ex}\label{localduality}
Let $\T$ be a big tensor-triangulated category and let $\mathscr{K}\subset \T^{\c}$ be a set of compact objects. Let $\T_{\Gamma}$ denote the $\mathscr{K}$-torsion objects of $\T$, that is 
\[
\T_{\Gamma}=\t{Loc}_{\T}^{\otimes}(\mathscr{K})
\]
is the smallest localising $\otimes$-ideal of $\T$ containing $\mathscr{K}$. Associated to $\T_{\Gamma}$ are two further subcategories, the $\mathscr{K}$-local objects $\T_{\K-\t{loc}}=(\T_{\Gamma})^{\perp_0}$ and the $\K$-complete objects $\T_{\Lambda}=(\T_{\K-\t{loc}})^{\perp_0}$, where for any subcategory $\X$ of $\T$, we define \[\X^{\perp_0} := \{Y \in \T : \Hom_\T(X, Y) = 0 \t{ for all $X \in \X$}\}.\] By \cite[2.21]{localduality} (see also~\cite[3.3.5]{axiomatic} and~\cite{cosupport}), the inclusion $\T_{\Gamma}\to \T$ admits a right adjoint, denoted $\Gamma$, while the inclusions $\T_{\K\t{-loc}}\to \T$ and $\T_{\Lambda}\to \T$ admit left adjoints, denoted $L$ and $\Lambda$ respectively. The functor $\Gamma$ is a colocalisation, while $L$ and $\Lambda$ are localisations.

We claim that the pair $(\T_{\Gamma},\T_{\Lambda})$ is a duality pair on $(\T,\T,\mbb{I})$. To see this, first note that any object $X\in \T$ appears in a triangle
\[
\Gamma X\to X\to LX,
\]
which induces a further triangle
\[
F(LX,\mbb{I})\to F(X,\mbb{I})\to F(\Gamma X,\mbb{I}).
\]

The functor $\Gamma\colon \T \to \T$ is left adjoint to $\Lambda$~\cite[2.21(4)]{localduality}, and this adjunction yields an isomorphism $F(\Gamma X,\mbb{I})\simeq \Lambda F(X,\mbb{I})$ which, combined with the above triangle shows that $F(X,\mbb{I})$ is complete (that is, in $\T_{\Lambda})$ if and only if $F(LX,\mbb{I})\simeq 0$.  By \cref{lem:BCproperties}, this is the case if and only if $LX\simeq 0$, which is equivalent to $X\in \T_{\Gamma}$. This shows the first condition. The second is trivial as $\T_{\Lambda}$ is closed under finite coproducts and retracts. As an immediate corollary of \cref{dp} we obtain that $\T_{\Gamma}$ is closed under pure subobjects, pure quotients and pure extensions, which to our knowledge is a new result.

We spell out some concrete instances of this example now. For more details on the functors $\Gamma$ and $\Lambda$ which arise in each of the following examples, together with some further examples of $\T$ and $\K$ one might choose, we refer the reader to~\cite{localduality}.
\begin{enumerate}
\item Let $R$ be a commutative ring and $I$ be a finitely generated ideal. Taking $\T = \msf{D}(R)$ and $\K = \{K(I)\}$ where $K(I)$ is the (unstable) Koszul complex on $I$, we obtain a duality pair $(\msf{D}(R)_\Gamma, \msf{D}(R)_\Lambda)$, where $\msf{D}(R)_\Gamma$ and $\msf{D}(R)_\Lambda$ consist of the derived $I$-torsion and derived $I$-complete complexes respectively. One can also generalize this to the setting where $R$ is a commutative ring spectrum and $I$ is a finitely generated ideal in $\pi_*R$.
\item Letting $\T = \mrm{Sp}$ be the stable homotopy category and $\K = \{F(n)\}$ where $F(n)$ is a finite type $n$ spectrum, we obtain a duality pair $(\C_{n-1}^f, L_{F(n)}\mrm{Sp})$ where $\C_{n-1}^f$ denotes the image of the acyclification with respect to $F(n)$, and $L_{F(n)}\mrm{Sp}$ denotes the Bousfield localisation at $F(n)$. Note that the resulting duality pair is independent of the choice of finite type $n$ spectrum $F(n)$. One may also perform a local version of this by setting $\T$ to be the category of $E(n)$-local spectra, where $E(n)$ denotes the Johnson-Wilson theory at some height $n$.
\item Take $\T = \mrm{Sp}_G$ to be the equivariant stable homotopy category with respect to a compact Lie group $G$, and let $\mc{F}$ be a family of subgroups of $G$. By taking $\K = \{G/H_+ \mid H \in \mc{F}\}$ one obtains a duality pair $(\mrm{Sp}_G^{\text{$\mc{F}$-free}}, \mrm{Sp}_G^{\text{$\mc{F}$-cofree}})$ where $\mrm{Sp}_G^{\text{$\mc{F}$-free}}$ denotes the subcategory of $\mc{F}$-free $G$-spectra (i.e., those which have geometric isotropy in $\mc{F}$, equivalently, those for which the natural map $E\mc{F}_+ \wedge X \to X$ is an equivalence) and $\mrm{Sp}_G^{\text{$\mc{F}$-cofree}}$ denotes the subcategory of $\mc{F}$-cofree $G$-spectra (i.e., those for which the natural map $X \to F(E\mc{F}_+, X)$ is an equivalence). 
\end{enumerate}
\end{ex}

\begin{ex}
Let $\T$ be a big tensor-triangulated category and fix a compact object $K\in\T$. Consider the Auslander and Bass classes associated to $K$:
\[
\A_K = \{X \in \T \mid X \xrightarrow{\sim} F(K, K \otimes X)\} \quad \mrm{and} \quad \B_K = \{Y \in \T \mid K \otimes F(K,Y) \xrightarrow{\sim} Y\}.
\]
We refer the reader to~\cite{foxbyequivalence} for an example of Auslander and Bass classes in relation to~\cref{localduality}, and for further references on the history of these classes. We can realise $\A_{K}$ (respectively $\B_{K}$) as the class of objects in $\T$ such that the unit (respectively counit) of the adjunction $K \otimes - \dashv F(K,-)$ on $\T$ is an equivalence. The pair $(\A_{K},\B_{K})$ is, in fact, a symmetric duality pair. To see this, suppose that $X\in\T$, then there are isomorphisms
\begin{align*}
K\otimes F(K, F(X,\mbb{I})) & \simeq K\otimes F(K\otimes X,\mbb{I}) \t{ by adjunction},\\
&\simeq F(F(K,K\otimes X),\mbb{I}) \t{ as $K$ is compact}.
\end{align*}
In particular, we see that if $X\in \A_{K}$, then $F(X,\mbb{I})\in\B_{K}$, while if $F(Y,\mbb{I})\in\B_{K}$, we must have $Y\in\A_{K}$. Clearly $\B_{K}$ is closed under retracts and finite coproducts. To see $(\B_{K},\A_{K})$ is a duality pair, note the following isomorphisms hold for any $Y\in\T$
\begin{align*}
	F(K,K\otimes F(Y,\mbb{I})) & \simeq F(K,F(F(K,Y),\mbb{I})) \t{ by the compactness of $K$,} \\
	& \simeq F(K\otimes F(K,Y),\mbb{I}) \t{ by adjunction}.
\end{align*} 
Again, we then see that $Y\in\B_{K}$ if and only if $F(Y,\mbb{I})\simeq F(K,K\otimes F(Y,\mbb{I}))$, that is $F(Y,\mbb{I})\in\A_{K}$. Again, $\A$ is closed under finite coproducts and retracts.
\end{ex}

\begin{ex}
Let $R$ be a commutative ring and $\mf{a}$ an ideal of $R$ generated by a sequence $x_{1},\cdots,x_{n}$ of elements of $R$. Recall that the $\mf{a}$-depth of a complex $M\in\D(R)$ is 
\[
\mf{a}\t{-depth}(M)= \inf\{n \in \mathbb{Z} : H_{-n}\msf{R}\Hom_R(R/\mf{a},M) \neq 0\}
\]
while the $\mf{a}$-width of $M$ is given by
\[
\mf{a}\t{-width}(M)=\inf\{n \in \mathbb{Z} : H_n(R/\mf{a} \otimes_R^\msf{L} M) \neq 0\}.
\]
For $i\geq 0$, let $\msf{D}_{i}$ denote the class of complexes of $\mf{a}$-depth at least $i$, and likewise define $\msf{W}_{i}$ to be the class of complexes of $\mf{a}$-width at least $i$. Then $(\msf{D}_{i},\msf{W}_{i})$ is a symmetric duality pair. This follows from \cite[2.5.7]{dcmca} or alternatively from \cite[14.4.14]{dcmca}. Note that, in the case of modules rather than complexes, it was shown in \cite[2.7]{hj} that there are analogous duality pairs.
\end{ex}

\begin{ex}
In~\cref{sec:silting} we apply duality pairs to study silting and cosilting classes, and we refer the reader there for more details. If $X\in\D(R)$ is bounded silting, then $X^{+}:=\msf{R}\Hom_{\Z}(X,\qz)$ is bounded cosilting in $\D(R^\circ)$ by~\cite[3.3]{HrbekHugel}. We show that if $\msf{Silt}(X)$ is the associated silting class of $X$ and $\msf{Cosilt}(X^+)$ is the associated cosilting class of $X^{+}$, then $(\msf{Silt}(X),\msf{Cosilt}(X^+))$ is a symmetric, (co)product closed duality pair. We moreover use this to identify the classes as dual definable categories, see \cref{siltingthm} for more details. 
\end{ex}

\section{Definability and duality pairs}
In this section, we explore the relationship between definable classes and duality pairs by considering Auslander--Gruson--Jensen duality for compactly generated triangulated categories. More precisely, we introduce an antiequivalence between left $\T^\c$-modules and left $\U^\c$-modules with respect to a suitable duality triple $(\T,\U,Q)$, which we call Auslander--Gruson--Jenson duality. We view this as a triangulated analogue of the classic Auslander--Gruson--Jensen duality for modules, which was independently introduced by Auslander in~\cite{Auslander} and Gruson--Jensen in~\cite{GrusonJensen}. In the classical setting of finitely accessible categories, this duality enables the introduction of dual definable categories~\cite{dac}, and we show that our antiequivalence allows the same in the triangulated setting.

\subsection{The case for modules}
To motivate this discussion, we provide some background on definability in module categories and then relate this to duality pairs. Even in this classical setting, we obtain a new result, see \cref{sym}. 

To this end, let $R$ be a ring and consider $(\mod{R},\tb{Ab})$, the category of additive functors from finitely presented $R$-modules to abelian groups. A functor $Q\in (\mod{R},\ab)$ is \ti{finitely presented} if there is a morphism $f\colon A\to B$ in $\mod{R}$ such that the induced sequence of functors
\[
\Hom_{R}(B,-)\to\Hom_{R}(A,-)\to Q\to 0
\] 
is exact in $(\mod{R},\ab)$. We will let $(\mod{R},\ab)^{\t{fp}}$ denote the category of finitely presented functors. As we can view $f\colon A\to B$ as a morphism in the category $\Mod{R}$, we can also consider the functor that is obtained as the cokernel of the induced map $\Hom_{R}(B,-)\to\Hom_{R}(A,-)$ in $(\Mod{R},\ab)$, which we shall denote by $\tilde{Q}$. By definition, the restriction of $\tilde{Q}$ to $\mod{R}$ is nothing other than $Q$, and as $A$ and $B$ are finitely presented modules, $\tilde{Q}$ commutes with direct limits and direct products. In fact, by \cite[\S 10.2.8]{psl}, $\tilde{Q}$ is the unique functor that extends $Q$ with these properties.

\begin{defn}
	A class $\mc{D}\subseteq\Mod{R}$ is said to be \ti{definable} if there is a set $\mc{S}\subset (\mod{R},\ab)^{\t{fp}}$ such that
	\[
	\mc{D}=\t{Ker}(\mc{S}):=\{M\in\Mod{R}:\tilde{Q}M=0 \t{ for all }Q\in\mc{S}\}.
	\]
\end{defn}

Equivalently, a class of modules is definable precisely when it is closed under direct limits, direct products and pure subobjects by~\cite[3.4.7]{psl}.

Note that given a definable class $\mathcal{D}$, the class of functors 
\[
\mathcal{S}_{\mathcal{D}}:=\{Q\in (\mod{R},\ab)^{\t{fp}}: Q(M)=0\t{ for all }M\in\mathcal{D}\}
\]
is a Serre subcategory. Conversely, given a Serre subcategory, one can consider the definable class given by its kernel. In this way one establishes a bijection between definable subcategories of $\Mod{R}$ and Serre subcategories of $(\mod{R},\ab)^{\t{fp}}$.  

To see how definable classes relate to duality pairs, we next recall the Auslander--Gruson--Jensen functor.
\begin{defn} 
The Auslander--Gruson--Jensen (AGJ) functor is a contravariant functor $\delta\colon(\mod{R},\ab)^{\t{fp}}\to (\mod{R^{\circ}},\ab)^{\t{fp}}$ defined by
\[
(\delta Q)(M) = \Hom_{(\mod{R},\ab)}(Q, M\otimes_{R}-)
\]
for every $Q\in (\mod{R},\ab)^{\t{fp}}$ and $M\in \mod{R^{\circ}}$.
\end{defn}
It can be seen in \cite[\S 10.3]{psl} that $\delta$ is an exact functor. Abusing notation, we can also define $\delta\colon (\mod{R^{\circ}},\ab)^{\t{fp}}\to (\mod{R},\ab)^{\t{fp}}$ via
\[
(\delta F)(N)=\Hom_{(\mod{R^{\circ}},\ab)}(F,-\otimes_{R}N)
\]
It follows that $\delta^{2}$ is naturally isomorphic to the identity functor. Consequently one obtains a bijection between Serre subcategories in $(\mod{R},\ab)^{\t{fp}}$ and $(\mod{R^{\circ}},\ab)^{\t{fp}}$, and thus there is a natural correspondence between definable subcategories of $\Mod{R}$ and $\Mod{R^{\circ}}$. 
\begin{defn}
For a definable category $\mathcal{D}$ of $\Mod{R}$ we define its \ti{dual definable category} $\mathcal{D}^d$ by $\mathcal{D}^{d}=\t{Ker}(\delta(\mathcal{S}_{\mathcal{D}}))$. 
\end{defn}
From the preceding discussion, it is immediate that $\mathcal{D}^{dd}=\mathcal{D}$. It is at this stage we return to duality pairs to highlight how, in practice, one can move between $\mathcal{D}$ and $\mathcal{D}^{d}$. 

\begin{thm}\label{sym}
Let $R$ be a ring and $\mathcal{D}$ be a definable category in $\Mod{R}$. Then $(\mathcal{D},\mathcal{D}^{d})$ is a symmetric, product-closed duality pair. 
Moreover, if $(\mathsf{A},\mathsf{B})$ is a symmetric duality pair, then the following are equivalent:
\begin{enumerate}
	\item either class is closed under coproducts;
	\item either class is closed under products;
	\item $\mathsf{A}$ is definable and $\mathsf{B}=\mathsf{A}^{d}$;
	\item $\mathsf{B}$ is definable and $\mathsf{A}=\mathsf{B}^{d}$.
\end{enumerate}
In particular, if either class is closed under (co)products, both classes are and they are thus definable.
\end{thm}
\begin{proof}
Write $(-)^+ = \Hom_{\mbb{Z}}(-,\qz)$ for the character dual. By~\cite[3.4.17]{psl} we have $M \in \mathcal{D}$ if and only if $M^{+} \in \mathcal{D}^{d}$ which shows that $(\mathcal{D}, \mathcal{D}^d)$ is a duality pair. It is product-closed by definition, and is symmetric since $\mathcal{D}^{dd} = \mathcal{D}$. We now show the equivalences. That $(1)$ and $(2)$ are equivalent follows from \cref{symmclosure}. For $(2)\implies (3)$, assume without loss of generality that $\mathsf{A}$ is closed under products. The class $\mathsf{A}$ is also closed under pure subobjects and pure quotients by~\cite[3.1]{hj} (see also \cref{hjtheorem}). Since direct limits are pure quotients of coproducts which are in turn pure subobjects of products, $\mathsf{A}$ is therefore closed under direct limits. Consequently $\mathsf{A}$ is definable and there is a symmetric duality pair $(\mathsf{A},\mathsf{A}^{d})$. If $B \in \mathsf{B}$, then $B^{+}\in \mathsf{A}$ and thus, by the first statement, $B^{++}\in\mathsf{A}^{d}$. Since $B \to B^{++}$ is a pure monomorphism, we have $B\in\mathsf{A}^{d}$. On the other hand, for $X\in\mathsf{A}^{d}$, we have $X^{++}\in\mathsf{B}$, hence $X\in\mathsf{B}$, yielding $\mathsf{A}^{d}=\mathsf{B}$. The implication $(3)\implies (4)$ is immediate as $\A^{dd} = \A$, and $(4)\implies (1)$ is trivial from the definition of definability.
\end{proof}

\begin{rem}\label{symmetry}
Note that the symmetric assumption on the duality pair in the previous theorem cannot be disposed of, even if $\mathsf{A}$ is definable. For example, let $R$ be a ring that is right noetherian. The class $\msf{F}$ of flat left $R$-modules is then definable as it is product closed (all this requires is right coherence), and the class $\msf{I}$ of injective right $R$-modules is also definable, as it is closed under directed colimits (see \cite[\S 3.4.3]{psl}). With these assumptions $(\msf{F},\msf{I})$ is a symmetric duality pair, but there is also a duality pair $(\msf{I},\msf{F}_{\t{PI}})$, where $\msf{F}_{\t{PI}}$ denotes the pure injective flat modules. This is because the dual of any module is pure injective, and the class of pure injective modules is closed under finite coproducts and direct summands. However $(\msf{F}_{\t{PI}},\msf{I})$ is, in general, not a duality pair. If it were, then by \cref{sym} $\msf{F}_{\t{PI}}$ would be definable, and since $\Hom_{\Z}(R,\qz)$ is an injective right $R$-module, it would follow that $R$ is a pure injective module over itself, and therefore every flat module would be pure injective. Therefore by \cite[3.2]{rothmaler}, $R$ would have to be a left perfect ring.
\end{rem}
However, with an additional assumption, the above theorem has the following corollary which enables us to bypass the concerns of the remark.

\begin{cor}\label{symmcor}
Suppose $(\A,\B)$ is a duality pair such that both $\A$ and $\B$ are definable. Then $\A=\B^{d}$, $\B = \A^d$, and the duality pair is symmetric.
\end{cor}

\begin{proof}
Let $B\in\B$, then, as $(\B, \B^d)$ is a symmetric duality pair by \cref{sym}, we have $B^{++}\in\B$. Yet as $(\A,\B)$ is a duality pair, it follows that $B^{+}\in\A$ and therefore $B^{++}\in\A^{d}$. Yet since $\A^{d}$ is definable, we have that $B\in\A^{d}$, hence $\B\subseteq \A^{d}$. For the other inclusion, if $Y\in\A^{d}$, then $Y^{+}\in \A$ and therefore $Y^{++}\in\B$, but, again, $Y\in\B$ as $\B$ is definable. Thus $\B=\A^{d}$ and the conclusion follows as $\A^{dd} = \A$.
\end{proof}

\subsection{The triangulated case}\label{ttcase}
Let us now turn to the analogous question for compactly generated triangulated categories. 
Firstly, we note that definability is well established in the setting of compactly generated triangulated categories (see, for example \cite{krcoh} for a treatment) and we recall the key points.
\begin{defn}
A class $\mc{D}$ of objects in a compactly generated triangulated category $\T$ is \ti{definable} if there is a set of maps $\{f_{i}\colon A_{i}\to B_{i}\}_{I}$, with $A_{i},B_{i}\in\T^{\c}$, such that
\[
\mc{D}=\{X\in\T:\Hom_{\T}(f_{i},X) \t{ is surjective for all }i\in I\}.
\]
\end{defn}
Note that this definition is essentially the same as the one for modules, with maps between compacts replacing maps between finitely presented modules. A functor $\T\to \ab$ that arises as 
\[
\t{coker}(\Hom_{\T}(f,-)\colon \Hom_{\T}(B,-)\to \Hom_{\T}(A,-))
\]
for a morphism $f\colon A\to B$ of compact objects is called a coherent functor in \cite{krcoh}, while its restriction to $\T^{\c}$ is called a \ti{finitely presented functor}. We let $(\T^{\c},\ab)^{\t{fp}}$ denote the category of finitely presented functors $\T^{\c}\to\ab$. 

As with definable classes of modules, definable classes in compactly generated triangulated categories (under a mild assumption) can be determined by closure properties. 
\begin{prop}\label{definableequivalent}
Let $\T$ be a compactly generated triangulated category which is the underlying category of a strong and stable derivator. Then the following conditions are equivalent for a class of objects $\mc{D}$ in $\T$:
\begin{enumerate}
\item $\mc{D}$ is definable;
\item $\mc{D}$ is closed under products, pure subobjects and directed homotopy colimits;
\item $\mc{D}$ is closed under products, pure subobjects and pure quotients.
\end{enumerate}
\end{prop}
\begin{proof}
(1) and (2) are equivalent by \cite[3.12]{laking}, and (1) implies (3) follows from the definition. So it remains to verify that (3) implies (2). For any directed category $I$, the natural map $\alpha\colon \oplus_I X_i \to \mrm{hocolim}_I X_i$ is a pure epimorphism, since for any compact object $C$, the map $\Hom_\T(C, \alpha)$ may be identified with \[\oplus_I \Hom_\T(C, X_i) \to \mrm{colim}_I\Hom_\T(C, X_i)\] by \cite[1.5]{Neemanconnection}. Therefore any directed homotopy colimit is a pure quotient of a coproduct, and since coproducts are pure subobjects of products, this completes the proof. 
\end{proof}

To link this to duality pairs, we now introduce a refined notion of duality triple. We note that this refined notion still encompasses the main examples of interest.
\begin{defn}\label{AGJdualitytriple}
An \emph{Auslander--Gruson--Jensen duality triple} is a duality triple $(\T,\U,Q)$ together with a pair of functors $D\colon (\T^\c)^\mathrm{op} \to \U^\c$ and $D\colon (\U^\c)^\mrm{op} \to \T^\c$ satisfying the following properties:
\begin{enumerate}
\item the functors \[\begin{tikzcd}\T^\c \arrow[r, yshift=1mm, "D"] & (\U^\c)^\mathrm{op} \arrow[l, yshift=-1mm, "D"] \end{tikzcd}\] form an adjoint equivalence of categories;
\item for all $C \in \T^\c$ and $X \in \T$, and all $C' \in \U^\c$ and $Y \in \U$ there are isomorphisms \[\Hom_\T(C,X)^+ = \Hom_\U(DC, QX) \quad \mathrm{and} \quad \Hom_\U(C',Y)^+ = \Hom_\T(DC', QY)\] which are natural in both $C$ and $X$, and $C'$ and $Y$, where $(-)^+ = \Hom_\Z(-, \qz)$ denotes the character dual.
\end{enumerate}
We write $(\T,\U,Q,D)$ for the data of an Auslander--Gruson--Jensen (AGJ) duality triple.
\end{defn}

\begin{rem}
We note that condition (1) amounts to saying that $D^2 = \mathrm{id}$, and that for all $X \in \T^\c$ and $Y \in \U^\c$, there is an isomorphism \[\Hom_\T(DY,X) = \Hom_\U(DX,Y)\] where everything is appropriately natural. As with the definition of duality triple, we note that the definition of AGJ duality triple is completely symmetric. 
\end{rem}

We now show that our two main duality pairs of interest form AGJ duality triples.
\begin{prop}\label{AGJdualitytripleexamples}\leavevmode
\begin{enumerate}
\item Let $\T$ be a big tensor-triangulated category. Then with $D = F(-, \1)$, the collection $(\T,\T,\mathbb{I},D)$ forms an AGJ duality triple. 
\item Let $R$ be a ring. Then with $D\colon \msf{D}(R) \to \msf{D}(R^\circ)$ (resp., $D\colon \msf{D}(R^\circ) \to \msf{D}(R))$ defined by $\msf{R}\Hom_R(-,R)$ (resp., $\msf{R}\Hom_{R^\circ}(-,R)$), the collection $(\msf{D}(R), \msf{D}(R^\circ), \msf{R}\Hom_\Z(-, \qz), D)$ forms an AGJ duality triple.
\end{enumerate}
\end{prop}
\begin{proof}
Firstly we consider the case of big tensor-triangulated categories. Here the adjoint equivalence is a well-known property of the compact objects, and the second condition follows from the definition of the Brown--Comenetz dual (see \cref{lem:BCproperties}). For the case of $\msf{D}(R)$, we note that there is a natural isomorphism $C\simeq\msf{R}\Hom_{R^{\circ}}(\msf{R}\Hom_{R}(C,R),R)$ for any $C \in \msf{D}(R)^\c$. Using this together with~\cite[7.5.27]{dcmca}, there are natural isomorphisms
\begin{align*}
\Hom_{\D(R)}(\msf{R}\Hom_{R^{\circ}}(Y,R),X) &\simeq \Hom_{\D(R)}(\msf{R}\Hom_{R^{\circ}}(Y,R),\msf{R}\Hom_{R^{\circ}}(\msf{R}\Hom_{R}(X,R),R)) \\
&\simeq \Hom_{\D(R^{\circ})}(\msf{R}\Hom_{R}(X,R),\msf{R}\Hom_{R}(\msf{R}\Hom_{R^{\circ}}(Y,R),R)) \\
&\simeq \Hom_{\D(R)}(\msf{R}\Hom_{R}(X,R),Y),
\end{align*}
which proves the first condition. For the second condition, we note that $H_0(X)^+ = H_0\msf{R}\Hom_\Z(X,\qz)$ for all $X \in \msf{D}(R)$ by definition. Now take $C \in \msf{D}(R)^\c$ and $X \in \msf{D}(R)$. By similar adjunction arguments to above we have
\begin{align*}
\Hom_{\msf{D}(R)}(C,X)^+ &= H_0\msf{R}\Hom_\Z(\msf{R}\Hom_R(C,X),\qz) \\
&= H_0\msf{R}\Hom_\Z(\msf{R}\Hom_R(C,R) \otimes^\msf{L}_R X, \qz) \\
&= H_0\msf{R}\Hom_{R^\circ}(\msf{R}\Hom_R(C,R), \msf{R}\Hom_\Z(X, \qz)) \\
&= \Hom_{\msf{D}(R^\circ)}(\msf{R}\Hom_R(C,R), \msf{R}\Hom_\Z(X,\qz))
\end{align*} as required.
\end{proof}

At this stage, we will fix an AGJ duality triple $(\T,\U,Q,D)$. We now define the Auslander--Gruson--Jensen (AGJ) duality in this context.
\begin{defn}
Define $\delta\colon (\T^{\c},\ab)^\t{op}\to(\U^{\c},\ab)$ via
\[
(\delta Q)(Y)=\Hom_{(\T^\c,\ab)}(Q,\Hom_{\T}(DY,-)).
\]
In an analogous way one also defines $\delta\colon (\U^\c, \ab)^\t{op} \to (\T^\c, \ab)$. 
\end{defn}
This functor will act as the Auslander--Gruson--Jensen transpose on our category, and the following theorem shows that it is a justifiable choice. We note that since the definition of an AGJ duality triple is symmetric in $\T$ and $\U$, all the statements below admit an alternative version with $\U$ taken to be the `base'.
\begin{thm}\label{agj}
The following statements hold:
\begin{enumerate}
	\item There is a natural isomorphism $\delta\Hom_{\T}(X,-) = \Hom_{\U}(DX,-)$ for any $X\in\T^{\c}$. In particular, $\delta^{2}\Hom_{\T}(X,-)=\Hom_{\T}(X,-)$.
	\item The functor $\delta$ restricts to an exact contravariant equivalence $(\T^{\c},\ab)^{\mrm{fp}}\to (\U^{\c},\ab)^{\mrm{fp}}$.
\end{enumerate}
\end{thm}

\begin{proof}
For the first item, suppose that $X$ is a compact object of $\T$. Then, for any compact $Y\in \U^{\c}$, there is a chain of isomorphisms:
\begin{align*}
\delta\Hom_{\T}(X,-)(Y) &= \Hom_{(\T^\c,\ab)}(\Hom_{\T}(X,-),\Hom_{\T}(DY,-)) \text{ by definition, } \\
 &\simeq \Hom_{\T}(DY,X) \text{ by Yoneda, } \\
 &\simeq \Hom_{\U}(DX,Y) \t{ by definition of AGJ duality triple,}
\end{align*}
which shows the first claim holds. 

For the second claim, we begin by showing that $\delta$ preserves finitely presented functors. Since $\delta$ is left exact (it is a hom functor), if $F\in(\T^{\c},\ab)^{\t{fp}}$ has a presentation $\Hom_{\T}(A,-)\to\Hom_{\T}(B,-)\to F\to 0$, where $A$ and $B$ are compact objects in $\T$, then applying $\delta$ gives an exact sequence
\[
0\to\delta F\to\delta\Hom_{\T}(B,-)\to\delta\Hom_{\T}(A,-).
\]
By the first claim, we know that both $\delta\Hom_{\T}(B,-)\simeq\Hom_{\U}(DB,-)$ and $\delta\Hom_{\T}(A,-)=\Hom_{\U}(DA,-)$ are finitely presented functors. But $(\U^{\c},\ab)^{\t{fp}}$ is an abelian category by \cite{krcoh}, hence $\delta F$ is also finitely presented.

We now show that the restriction of $\delta$ to finitely presented objects is exact. Since $\delta$ is left exactness, by its definition it suffices to show that $\t{Ext}_{(\U^\c,\ab)}^{1}(Q,\Hom_{\U}(DX,-))=0$ for any object $X\in\T^{\c}$ and $Q$ finitely presented. Since, $\Hom_{\U}(DX,-)$ is a cohomological functor, it is fp-injective by \cite[2.7]{krsmash}, hence this Ext group vanishes by the assumption on $Q$. That $\delta$ is an equivalence follows immediately from part (1), exactness, and the fact that $D^2 = \mrm{id}$.
\end{proof}

The following corollary provides a useful equivalent formulation of definability.
\begin{cor}\label{kerdef}
A full subcategory $\mc{D}\subseteq\T$ is definable if and only if there is a set of maps $\{g_{i}\}_{I}$ between compact objects such that
\[
\mc{D}=\{X\in\T:\mrm{Ker}(\Hom_{\T}(g_{i},X))=0 \t{ for all }i\in I\}.
\]
\end{cor}

\begin{proof}
If $F$ is a finitely presented functor, then so is $\delta F$ by ~\cref{agj}. Hence there is a map $g\in\U^{\c}$ such that $\delta F=\t{Coker}\,\Hom_{\U}(g,-)$. By dualising this, it is clear from \cref{agj} that $F=\delta^{2}F=\t{Ker}\,\Hom_{\T}(Dg,-)$.
\end{proof}

As in the module case, there is a bijection between Serre subcategories of $(\T^{\c},\ab)^{\t{fp}}$ and definable subcategories of $\T$, which is described in \cite{krcoh}. We are similarly able to define the dual definable category of $\mc{D}$; we let $\mc{S}_{\mc{D}}$ denote the unique Serre subcategory corresponding to $\mc{D}$, and define
\begin{equation}\label{defndualdef}
\mc{D}^{d}=\mrm{Ker}(\delta \mc{S}_{\mc{D}})
\end{equation}
which is a definable subcategory of $\U$.
Using \cref{agj}, we can identify the functors which define the dual definable category.

\begin{lem}\label{dualiskernels}
	Let $\Phi=\{f_{i}\colon A_{i}\to B_{i}\}_{i \in I}$ be a set of morphisms between compact objects in $\T$, and let $\mc{D}=\{X\in\T:\mrm{Coker}(f_{i},X)=0 \t{ for all }i \in I\}$. Then
	\[
	\mc{D}^{d}=\{Y\in\U : \mrm{Ker}(Df_{i},Y)=0 \t{ for all }i \in I\}.
	\]
\end{lem}

\begin{proof}
	In $(\T^{\c},\ab)^{\t{fp}}$ there is, for each $i\in I$, an exact sequence
	\[
	\Hom_{\T}(B_{i},-)\xrightarrow{\Hom_{\T}(f_{i},-)} \Hom_{\T}(A_{i},-)\to \mrm{Coker}\Hom_{\T}(f_{i},-)\to 0.
	\]
	Applying $\delta$, which, by ~\cref{agj}, is exact, one obtains the exact sequence
	\[
	0\to \delta\mrm{Coker}\Hom_{\T}(f_{i},-)\to\Hom_{\U}(DA_{i},-)\xrightarrow{\Hom_{\U}(Df_{i},-)}\Hom_{\U}(DB_{i},-);
	\]
	showing that $\delta \mrm{Coker}\Hom_{\T}(f_{i},-)=\mrm{Ker}(\Hom_{\U}(Df_{i},-))$. The claim now follows from the definition of the dual definable category.
\end{proof}

We are now in a position to see how duality pairs relate to definable categories. In a direct analogy to \cref{sym}, we obtain the following results.

\begin{thm}\label{dualdefinable}
	Let $(\T,\U,Q,D)$ be an AGJ duality triple. If $\mc{D}$ is a definable subcategory of $\T$, then $(\mc{D},\mc{D}^{d})$ is a symmetric product closed duality pair on the duality triple $(\T, \U, Q)$.
\end{thm}
\begin{proof}
	Firstly, we must show that $X\in \mathcal{D}$ if and only if $QX\in \mathcal{D}^{d}$. So, let $\mc{D}$ be definable with representative set of maps $\{f_i\colon A_i \to B_i\}$. Then
	\begin{align*}
		X\in \mc{D} & \iff \Hom_{\T}(B_i,X)\to \Hom_{\T}(A_i,X)\to 0 \t{ is exact for all $i$}\\
		& \iff 0 \to \Hom_{\T}(A_i,X)^+ \to \Hom_{\T}(B_i,X)^+ \t{ is exact for all $i$}.
	\end{align*}
	By definition of an AGJ duality triple, we have $\Hom_{\T}(C,X)^{+}=\Hom_{\U}(DC,QX)$ for any $C \in \T^\c$. Therefore we see that $X \in \mc{D}$ if and only if
 \[0\to \Hom_{\U}(DA_i, QX)\to \Hom_{\U}(DB_i, QX)\]
is exact for all $i$,
	by compactness of $A_i$ and $B_i$. This last equivalence is the same as saying that $QX\in\mc{D}^d$ by ~\cref{dualiskernels}. Since definable classes are always closed under finite sums and retracts, it follows that $(\mc{D},\mc{D}^{d})$ is a duality pair. This is moreover a symmetric duality pair as $\mc{D}^{dd}=\mc{D}$. Last but not least, the classes being product closed follows trivially from the definition of definability.
\end{proof}

We firstly state the following corollary which gives a triangulated version of ~\cref{symmcor}. The proof is identical.
\begin{cor}\label{definablesymmetry}
	Let $(\T,\U,Q,D)$ be an AGJ duality triple, and suppose $(\A,\B)$ is a duality pair on the duality triple $(\T, \U, Q)$. If both $\A$ and $\B$ are definable, then $\A=\B^{d}$, $\B=\A^{d}$ and the duality pair is symmetric.
\end{cor}

We can now give the analogue of the enumerated part of \cref{sym}.
\begin{cor}\label{definablett}
	Let $(\T,\U,Q,D)$ be an AGJ duality triple, and suppose $(\A,\B)$ is a symmetric duality pair on the duality triple $(\T, \U, Q)$. Consider the following conditions:
	\begin{enumerate}
		\item either class is closed under coproducts;
		\item either class is closed under products;
		\item $\mathsf{A}$ is definable and $\mathsf{B}=\mathsf{A}^{d}$;
		\item $\mathsf{B}$ is definable and $\mathsf{A}=\mathsf{B}^{d}$.
	\end{enumerate}
The implications \[\begin{tikzcd}
	(1) \arrow[r, Leftrightarrow] & (2) \\
	(3) \arrow[r, Leftrightarrow] \arrow[u, Rightarrow] & (4) \arrow[u, Rightarrow]
\end{tikzcd} \] always hold. If $\T$ and $\U$ are the underlying categories of strong and stable derivators,
then all four statements are equivalent.
\end{cor}

\begin{proof}
The equivalence of $(1)$ and $(2)$ follows from \cref{symmclosure}. The equivalence of $(3)$ and $(4)$, and the implication $(4)\implies (2)$ are trivial. 

We now suppose that $\T$ and $\U$ are the underlying categories of strong and stable derivators, and show $(2)\implies (3)$. By \cref{definableequivalent}, being definable is equivalent to being closed under products, pure subobjects and pure quotients. By ~\cref{dp} we know $\A$ is closed under pure subobjects and pure quotients. Since $\A$ is closed under products by assumption, it is definable. It also follows that $\B$ is definable, and therefore $\B=\A^{d}$ by \cref{definablesymmetry}.
\end{proof}

\begin{cor}\label{triiffdual}
Let $(\T,\U,Q,D)$ be an AGJ duality triple. If $\mc{D}\subseteq\T$ is definable, then $\mc{D}$ is closed under extensions if and only if $\mc{D}^{d}$ is, and $\mc{D}$ is triangulated if and only if $\mc{D}^{d}$ is. 
\end{cor}
\begin{proof}
Combining \cref{dualdefinable} with \cref{closure} yields the result.
\end{proof}

\begin{rem}
In the preparation of this document, Rose Wagstaffe was kind enough to share her doctoral thesis with us. We noticed that she also has a proof of the above statement at \cite[8.1.15]{wagstaffe}, although her techniques are completely independent of ours.
\end{rem}

We end this section by using the results of \cref{minimaldualitypair} to give explicit constructions of definable closures in certain cases. Given any class $\mc{S}\subseteq\T$, one can form the definable closure of $\mc{S}$, denoted $\t{Def}(\mc{S})$, which is the smallest definable subclass of $\T$ containing $\mc{S}$. When $\T$ is the underlying category of a strong and stable derivator, \cref{definableequivalent} shows that this is the closure of $\mc{S}$ under products, pure subobjects and pure quotients. Since taking the duality class generated by $\mc{S}$ is closed under pure subobjects and pure quotients by \cref{dp}, it is enough to worry about products. Recall that for any class of objects $\mc{C}$ in $\T$ we define $\overline{\mc{C}}=\{X\in\T:QX\in\t{add}(Q\mc{C})\}$. 

\begin{prop}\label{defclosure}
	Let $(\T, \msf{U}, Q, D)$ be an AGJ duality triple and suppose that $\T$ is the underlying category of a strong and stable derivator. Let $\mc{S} \subseteq \T$. Then the definable closure $\mrm{Def}(\mc{S})$ of $\mc{S}$, is equal to $\overline{\mrm{Prod}(\mc{S})}$.
\end{prop}

\begin{proof}
	Let us first observe that $\t{Def}(\mc{S})$ is equal to $\t{Prod}(\mc{S})^{\t{pq,ps}}$ for the same reason as the module case (see \cite[p.109]{psl}). We therefore show that $\t{Prod}(\mc{S})^{\t{pq,ps}}=\overline{\t{Prod}(\mc{S})}$.

	To show that $\t{Prod}(\mc{S})^{\t{pq,ps}} \subseteq \overline{\mrm{Prod}(\mc{S})}$ it suffices to show that $\overline{\mrm{Prod}(\mc{S})}$ contains $\mrm{Prod}(\mc{S})$ and is closed under pure quotients and pure subobjects. It is clear that it contains $\mrm{Prod}(\mc{S})$ so suppose that $Y$ is either a pure subobject or a pure quotient of $X \in \overline{\mrm{Prod}(\mc{S})}$. By~\cref{detects}, it follows that $QY$ is a retract of $QX$, and hence $Y \in \overline{\mrm{Prod}(\mc{S})}$ as required.
	
	For the reverse inclusion if $X\in\overline{\t{Prod}(\mc{S})}$, then by definition, $QX$ is a retract of an object of the form $Q(\prod_{I}S_{i})$ where each $S_i \in \mc{S}$, as $Q$ commutes with finite coproducts. As $\t{Def}(\mc{S})$ is definable, there is a symmetric duality pair $(\t{Def}(\mc{S}),\t{Def}(\mc{S})^{d})$ by \cref{dualdefinable}, where $\t{Def}(\mc{S})^{d}$ is the dual definable category of $\t{Def}(\mc{S})$. In particular, we have $\prod_{I}S_{i}\in\t{Def}(\mc{S})$, and therefore $Q(\prod_{I}S_{i})\in\t{Def}(\mc{S})^{d}$. But then $QX$ is also in $\t{Def}(\mc{S})^{d}$, as definable categories are closed under retracts. But since $(\t{Def}(\mc{S}),\t{Def}(\mc{S})^{d})$ is a duality pair, we see that $X$ must also be an object in $\t{Def}(\mc{S})$ as required.
\end{proof}

\begin{cor}
	Let $(\T, \msf{U}, Q, D)$ be an AGJ duality triple and suppose that $\T$ is the underlying category of a strong and stable derivator. Let $\mc{S} \subseteq \T$ be a class of objects closed under products and retracts. Then $\mrm{Def}(\mc{S})=\overline{\mc{S}}$.
\end{cor}

\section{Application to silting}\label{sec:silting}
In this section, we give an application of our results to representation theory, namely to the study of bounded silting and cosilting objects in the derived category of a ring $R$. Recall that if $X$ is an object and $I$ is an indexing set, we write $X^{(I)}$ (resp., $X^{I}$) for the the $I$-indexed coproduct (resp., $I$-indexed product) of $X$.

\begin{defn}\label{siltingdef}
A complex $X\in\mathsf{D}(R)$ is a \ti{bounded silting object} if $\Hom_{\msf{D}(R)}(X,\Sigma^i X^{(J)})=0$ for all sets $J$ and integers $i>0$, and $\t{thick}(\t{Add}(X))=\msf{K}^b(\t{Proj}(R))$. The class \[X^{\perp_{>0}} := \{Y \in \msf{D}(R) : \Hom_{\msf{D}(R)}(X,\Sigma^i Y) = 0 \t{ for all $i > 0$}\}\] is called the associated \ti{silting class}, which we will denote by $\msf{Silt}(X)$.
\end{defn}

\begin{rem}\label{siltingdefn}
The previous definition is equivalent to saying that $(X^{\perp_{>0}}, X^{\perp_{<0}})$ is a t-structure in $\msf{D}(R)$ and $X \in \msf{K}^b(\t{Proj}(R))$. This follows from the equivalence between (1) and (3) in~\cite[4.2]{AMV}; also see~\cite[5.3]{Hugelsurvey}.
\end{rem}

Given an arbitrary class $\mc{X}\subset \msf{D}(R)$, necessary and sufficient conditions are known to determine whether or not $\mc{X}$ is a silting class, see \cite[3.6]{MarksVitoria}. Although the definability of silting classes is implicit in \cite[3.6]{MarksVitoria}, we state a proof as we shall require this later in the section.

\begin{lem}\label{siltingisdefinable}
Any silting class is definable.
\end{lem}

\begin{proof}
	By~\cite[3.6]{MarksVitoria} it is enough to show that $\mc{S}^{\perp_{>0}}$ is definable for any set of compact objects $\mc{S}$. If $S\in\mc{S}$, then $\Hom_{\D(R)}(S,-)$ is a finitely presented functor, hence $\mc{K}_{S}=\t{Ker}\,\Hom_{\D(R)}(S,-)$ is definable. Since $\Hom_{\D(R)}(S,\Sigma^iX)\simeq\Hom_{\D(R)}(\Sigma^{-i}S,X)$, and shifts of compacts are compact, we also see $\mc{K}_{\Sigma^{-i}S}$ is definable for every $i\geq 0$. As arbitrary intersections of definable classes are definable, it follows that 
	\[
	\mc{S}^{\perp_{>0}}=\bigcap_{S\in\mc{S},i>0}\mc{K}_{\Sigma^{-i}S}
	\]
	is definable.
\end{proof}

The dual notion to silting is cosilting. 
\begin{defn}
A complex $Z\in\msf{D}(R)$ is a \ti{bounded cosilting object} if $\Hom_{\msf{D}(R)}(Z^{J},\Sigma^iZ)=0$ for all sets $J$ and integers $i>0$, and $\t{thick}(\t{Prod}(Z))=\msf{K}^{b}(\t{Inj}(R))$. The class \[^{\perp_{>0}}Z:= \{Y \in \msf{D}(R) : \Hom_{\msf{D}(R)}(Y, \Sigma^i Z) = 0 \t{ for all $i > 0$}\}\] is called the associated \emph{cosilting class}, which we will denote by $\msf{Cosilt}(Z)$.
\end{defn}

Dually to \cref{siltingdefn}, a complex $Z \in \msf{D}(R)$ is bounded cosilting if and only if $(^{\perp_{<0}}Z,^{\perp_{>0}}\!Z)$ is a t-structure in $\msf{D}(R)$ and $Z \in \msf{K}^b(\t{Inj}(R))$. A sketch proof of this may be found in the discussion after~\cite[6.8]{Hugelsurvey}. As with the silting case, necessary and sufficient conditions for a class of objects in $\D(R)$ to be a cosilting class can be found in \cite[3.14]{MarksVitoria}. Of particular note to us is that every cosilting class is definable.

Throughout this section we write $(-)^+$ for the derived character dual $\msf{R}\Hom_{\Z}(-, \qz)$. Before we can state and prove our two main results in this section, we record the following easy lemma.
\begin{lem}\label{characteradjunctionswap}
Let $X \in \msf{D}(R)$ and $Y \in \msf{D}(R^\circ)$. Then $\msf{R}\Hom_R(X, Y^+) \simeq \msf{R}\Hom_{R^\circ}(Y, X^+).$
\end{lem} 
\begin{proof}
We have 
\[\msf{R}\Hom_R(X, Y^+) \simeq \msf{R}\Hom_\Z(Y \otimes_R^\msf{L} X, \qz) \simeq \msf{R}\Hom_\Z(X \otimes_{R^\circ}^\msf{L} Y, \qz) \simeq \msf{R}\Hom_{R^\circ}(Y, X^+)\] using standard isomorphisms (e.g.,~\cite[7.5.14 and 7.5.34]{dcmca}). 
\end{proof}

We now give our first main application of duality pairs to silting theory. We recall from \cref{AGJdualitytripleexamples} that the duality triple $(\msf{D}(R), \msf{D}(R^\circ), (-)^+)$ has an enhancement to an AGJ duality triple.
\begin{prop}\label{dualiscosilting}
Let $R$ be a ring. If $X$ is a bounded silting object in $\msf{D}(R)$ then $X^+$ is a bounded cosilting object in $\D(R^{\circ})$.
\end{prop}

\begin{proof}
	We first show the vanishing condition. Note the following isomorphism
	\[
	\Hom_{\msf{D}(R)}(X,\Sigma^iX^{(J)})\simeq H_{-i}\msf{R}\Hom_{R}(X,X^{(J)}).
	\]
	There are also isomorphisms
	\begin{align*}
		\msf{R}\Hom_{R^\circ}((X^{+})^{J},X^{+})
		&\simeq \msf{R}\Hom_{R}(X,((X^{+})^{J})^{+}) \t{ by \cref{characteradjunctionswap},} \\
		& \simeq \msf{R}\Hom_{R}(X,(X^{(J)})^{++}) \t{ as $(X^{+})^{J}\simeq (X^{(J)})^{+}$. }
	\end{align*}
	These isomorphisms show that the vanishing condition amounts to checking that $(X^{(J)})^{++}$ is in the silting class $\msf{Silt}(X) = X^{\perp_{>0}}$. The silting class $\msf{Silt}(X)$ is definable by \cref{siltingisdefinable}, and therefore $A\in \msf{Silt}(X)$ if and only if $A^{++}\in\msf{Silt}(X)$ by \cref{dualdefinable}. In particular, as $X^{(J)} \in \msf{Silt}(X)$ by definition of a bounded silting object, we have that $(X^{(J)})^{++}\in\msf{Silt}(X)$ as required.
	
	We now show the thickness condition. To show that $\mrm{thick}(\mrm{Prod}(X^+)) \subseteq \msf{K}^b(\mrm{Inj}(R^\circ))$ it suffices to show that $\msf{K}^b(\mrm{Inj}(R^\circ))$ contains $\mrm{Prod}(X^+)$ and is thick. It is clear that it is thick, so suppose that $A\in\mrm{Prod}(X^{+})$. Then $A$ is a retract of $\prod_I X^+\simeq (\oplus_I X)^+$. But $\oplus_I X$ is a bounded complex of projective modules, by the assumption that $X$ is silting, and therefore $\prod_I X^{+}\in \msf{K}^{b}(\t{Inj}(R^\circ))$, and hence $A$ is in $\msf{K}^b(\mrm{Inj}(R^\circ))$.
	
	For the other inclusion, we firstly note that $\msf{K}^b(\mrm{Inj}(R^\circ)) = \mrm{thick}(\mrm{Prod}(E))$ where $E$ is an injective cogenerator for $\Mod{R^\circ}$. (Indeed, the reverse inclusion is clear as $\msf{K}^b(\mrm{Inj}(R^\circ))$ is thick and contains $\mrm{Prod}(E)$, and the forward inclusion follows by induction using brutal truncations.) Therefore to prove the inclusion $\msf{K}^b(\mrm{Inj}(R^\circ)) \subseteq \mrm{thick}(\mrm{Prod}(X^+))$, it suffices to show that $R^+ \in \mrm{thick}(\mrm{Prod}(X^+))$, as $R^+$ is an injective cogenerator for $\Mod{R^\circ}$. As $X$ is bounded silting, $\mrm{thick}(\mrm{Add}(X)) = \msf{K}^b(\mrm{Proj}(R))$, so $R \in \mrm{thick}(\mrm{Add}(X))$ and hence $R^+ \in \mrm{thick}(\mrm{Prod}(X^+))$ as required.
\end{proof}

\begin{rem}
	\Cref{dualiscosilting} gives an alternative proof of part of~\cite[3.3]{HrbekHugel} using completely different machinery. 
\end{rem}

Next we give another application of duality pairs to silting, which allows us to identify the silting class of a bounded silting object and the cosilting class of its derived character dual as a dual definable pair.
\begin{thm}\label{siltingthm}
	Let $R$ be a ring, and let $X$ be a bounded silting object in $\D(R)$. Then $(\msf{Silt}(X), \msf{Cosilt}(X^+))$ is a symmetric duality
	pair. Moreover, $\msf{Silt}(X)^d = \msf{Cosilt}(X^+)$ and $\msf{Cosilt}(X^+)^d = \msf{Silt}(X)$.
\end{thm}
\begin{proof}
By~\cref{dualiscosilting}, $X^{+}$ is a bounded cosilting object. By~\cref{siltingisdefinable} and \cite[3.14]{MarksVitoria}, both $\msf{Silt}(X)$ and $\msf{Cosilt}(X^+)$ are definable. Therefore, by \cref{definablesymmetry} it suffices to show that $(\msf{Cosilt}(X^+),\msf{Silt}(X))$ is a duality pair. Suppose that $Y \in \D(R^\circ)$. Then by \cref{characteradjunctionswap} \[\msf{R}\Hom_{R^{\circ}}(Y, X^+) \simeq \msf{R}\Hom_{R}(X, Y^+)\] so, by taking homology and appropriate shifts, we see that $Y \in \msf{Cosilt}(X^+)$ if and only if $Y^+ \in \msf{Silt}(X)$, so $(\msf{Cosilt}(X^+),\msf{Silt}(X))$ is a duality pair as required.
\end{proof}

\section{Application to definability in stratified tt-categories}\label{sec:stratified}
In this section we explore the Auslander--Gruson--Jensen duality operation on definable subcategories in big tensor-triangulated categories. In particular, we show that definable subcategories of big tensor-triangulated categories are self-dual remarkably often. Throughout this section, we use the fact that the duality triple $(\T,\T,\mathbb{I})$ for $\T$ a big tensor-triangulated category admits an enhancement to an AGJ duality triple as shown in \cref{AGJdualitytripleexamples}.

In \cref{triiffdual} we showed that a definable subcategory is triangulated if and only if its dual definable category is. We now prove the analogous result for closure under tensor products.

\begin{prop}\label{idealdefinable}
Let $\T$ be a big tensor-triangulated category and $\mc{D}$ be a definable subcatgory of $\T$. Then $\mc{D}$ is a $\otimes$-ideal if and only if $\mc{D}^{d}$ is a $\otimes$-ideal.
\end{prop}

\begin{proof}
As $\mc{D}^{dd}=\mc{D}$, it suffices to prove one implication, so we assume that $\mc{D}^d$ is a $\otimes$-ideal. By \cite[5.1.11]{wagstaffe}, it suffices to show that $\mc{D}$ is closed under tensoring with compacts. Thus, let $C\in\T^{\c}$ and $X\in\mc{D}$. By \cref{dualdefinable}, we have $C\otimes X\in\mc{D}$ if and only if $\mbb{I}(C\otimes X)\in\mc{D}^{d}$. Yet
\[
\mbb{I}(C\otimes X) \simeq F(C\otimes X,\mbb{I}) \simeq DC\otimes \mbb{I}X.
\] 
As $X\in\mc{D}$, we know $\mbb{I}X\in\mc{D}^{d}$, and, by assumption, $\mc{D}^{d}$ is a $\otimes$-ideal, which completes the proof.
\end{proof}

We now turn to understanding self-duality of definable classes. In order to apply the symmetric duality pair $(\mc{D}, \mc{D}^d)$ proved in \cref{dualdefinable} we firstly need to understand how the Brown--Comenetz dual functor acts on a certain important kind of definable class. For a set of objects $\msf{X}$, we write $\msf{X}^\perp = \{Y \in \T: F(X, Y) = 0 \t{ for all $X \in \X$}\}$.
\begin{lem}\label{lem:closedunderI}
	Let $\T$ be a big tensor-triangulated category and $\mathsf{X}$ be a set of compacts in $\T$. For any $Y \in \msf{X}^\perp$, we have $X \otimes Y \simeq 0$ for all $X \in \msf{X}$. In particular, $\mathsf{X}^\perp$ is closed under the Brown--Comenetz dual functor $\mbb{I}$. 
\end{lem}
\begin{proof}
	Let $X \in \mathsf{X}$ and $Y \in \mathsf{X}^\perp$. Since $X$ is a retract of $X \otimes DX \otimes X$, we have that $X \otimes Y$ is a retract of $X \otimes DX \otimes X \otimes Y \simeq X \otimes F(X,Y) \otimes X \simeq 0$, and hence is zero. The second claim follows since $F(X, \bc{Y}) \simeq F(X \otimes Y, \mbb{I})$.
\end{proof}

We can now prove our first result about self-duality of definable subcategories. Combined with the previous lemma, it shows that in big tensor-triangulated categories which are generated by their tensor unit (such as the derived category of a commutative ring), the right orthogonal to sets of compacts is always a self-dual definable category.
\begin{prop}\label{selfdualperp}
	Let $\T$ be a big tensor-triangulated category and $\mathsf{X}$ be a set of compact objects in $\T$. Then $\mathsf{X}^\perp$ is a definable subcategory with the property that $(\mathsf{X}^\perp)^d = \mathsf{X}^\perp$.
\end{prop}
\begin{proof}
	That $\mathsf{X}^\perp$ is definable is clear by similar reasoning to that of \cref{siltingisdefinable}. By \cref{dualdefinable} we have a symmetric duality pair $(\mathsf{X}^\perp, (\mathsf{X}^\perp)^d)$. We now show that $\mathsf{X}^\perp \subseteq (\mathsf{X}^\perp)^d$. Suppose that $Y \in \mathsf{X}^\perp$. Therefore $\mbb{I}Y \in \mathsf{X}^\perp$ by \cref{lem:closedunderI} and hence $\mbb{I}^2Y \in (\mathsf{X}^\perp)^d$ since $(\mathsf{X}^\perp, (\mathsf{X}^\perp)^d)$ is a duality pair. By \cref{symmclosure}, we therefore have $Y \in (\mathsf{X}^\perp)^d$ showing the desired inclusion. For the reverse inclusion, if $Y \in (\mathsf{X}^\perp)^d$, then $\mbb{I}Y \in \mathsf{X}^\perp$ and hence $\mbb{I}^2Y \in \mathsf{X}^\perp$ by \cref{lem:closedunderI}. Therefore $Y \in \mathsf{X}^\perp$ by \cref{symmclosure} and $(\mathsf{X}^\perp)^d = \mathsf{X}^\perp$ as claimed.
\end{proof}

We now turn to generalizing the previous result to more general definable subcategories. There is an associated cost to this in that we must ask our tensor-triangulated category $\T$ to be canonically stratified in the sense of~\cite{bik2}, but this holds in many cases of interest. Therefore we briefly recall what it means for a big tensor-triangulated category $\T$ to be canonically stratified now, and refer the reader to \cite{bik3,bik2} for more details. Suppose that $\T$ has an action of a graded commutative Noetherian ring $R$, that is, compatible maps $R \to \mrm{End}(X)$ to the graded endomorphisms of each $X \in \T$, which is canonical in the sense that it factors over $\mrm{End}(\1)$. Using the action of $R$ on $\T$ one may construct a local cohomology functor $\Gamma_\p\colon \T \to \T$ for each $\p \in \mrm{Spec}(R)$, and this gives a support theory on $\T$ via $\mrm{supp}(X) = \{\p \in \mrm{Spec}(R) \mid \Gamma_\p X \not\simeq 0\}$. One may extend this to subcategories of $\T$ by taking the union of the supports of objects in the subcategory, that is, for a subcategory $\mathcal{S}$ of $\T$ we define \[\mathrm{supp}(\mathcal{S}) = \bigcup_{X \in \mathcal{S}} \mathrm{supp}(X).\] The category $\T$ is said to be \emph{canonically stratified} by $R$ if $\Gamma_\p\T$ is minimal for each $\p \in \mrm{Spec}(R)$; that is, $\Gamma_\p\T$ contains no non-zero proper localising $\otimes$-ideals. Note that the local-to-global principle automatically holds for localising $\otimes$-ideals. Stratification is moreover equivalent to the map
\[\mrm{supp}\colon \{\text{localising $\otimes$-ideals of $\T$}\} \to \{\text{subsets of $\text{supp}(\T)$}\}\] being a bijection. Many big tensor-triangulated categories of interest are canonically stratified in this sense; for example, the derived category $\msf{D}(R)$ of a commutative Noetherian ring is canonically stratified by $R$~\cite{Neemanchrom}, the stable module category $\mrm{StMod}(kG)$ of a finite group is canonically stratified by $H^*(G;k)$~\cite{bik}, the derived category $\msf{D}(A)$ of a commutative Noetherian DGA is stratified by $H^*A$ if $A$ is formal~\cite[8.1]{bik2} or if it non-positive and has finite amplitude~\cite[4.11]{ShaulWilliamson}, as well as many other interesting examples arising from topology. Recall that if the tensor unit generates, then localising $\otimes$-ideals are the same as localising subcategories.

We are now in a position to state our next main theorem about self-duality of definable classes. The following theorem extends~\cite[11.8]{bik} by providing some more equivalent conditions to those presented there. For a subcategory $\mathsf{X}$ of $\T$, we set $\mathsf{X}^{\perp_{\Z}} := \{Y \in \T : \Hom_\T(X, \Sigma^i Y) = 0 \t{ for all $X \in \mathsf{X}$ and $i \in \mathbb{Z}$}\}.$

\begin{thm}\label{stratequiv}
	Let $\T$ be a big tensor-triangulated category which is canonically stratified, and $\mathsf{L}$ be a localising $\otimes$-ideal of $\T$. The following are equivalent:
	\begin{enumerate}
		\item $\mathsf{L}$ is product closed;
		\item the complement of $\mrm{supp}(\mathsf{L})$ in $\mrm{supp}(\T)$ is specialization closed;
		\item $\mathsf{L}$ is closed under the Brown--Comenetz dual functor $\mbb{I}$;
		\item $\mathsf{L} = \mathsf{X}^{\perp_{\Z}}$ for some set of compacts $\mathsf{X}$;
		\item $\mathsf{L} = \mathsf{X}^\perp$ for some set of compacts $\mathsf{X}$;
		\item $\mathsf{L}$ is definable and $\mathsf{L}^d = \mathsf{L}$;
		\item $\mathsf{L}$ is definable and closed under the Brown--Comenetz dual functor $\mbb{I}$;
		\item $\mathsf{L}$ is definable.
	\end{enumerate}
\end{thm}
\begin{proof}
	That (1), (2), (3), (4) are equivalent is~\cite[11.8]{bik}. They state the result only for the categories $\mathrm{KInj}(kG)$ and $\mrm{StMod}(kG)$, but the proof they give holds for any big tensor-triangulated category which is canonically stratified. The key ingredients are the existence of objects $T(I) \in \T$ lifting injectives which always exist by Brown representability, and the bijection \[\mrm{supp}\colon \{\text{localising $\otimes$-ideals of $\T$}\} \xrightarrow{\cong} \{\text{subsets of $\text{supp}(\T)$}\}\] given by the stratification hypothesis. For (4) implies (5), it suffices to note that if $\mathsf{X}^{\perp_{\Z}}$ is a $\otimes$-ideal, then $\mathsf{X}^{\perp_{\Z}} = \mathsf{X}^\perp$. To prove this, if $\Hom_\T(X, Y) = 0$ for all $X \in \mathsf{X}$, then $\Hom_\T(C, F(X,Y)) = 0$ for all $C \in \T^\c$ and $X \in \mathsf{X}$, by adjunction and the $\otimes$-ideal assumption. Hence $F(X,Y) = 0$ for all $X \in \mathsf{X}$ as required. We have (5) implies (6) by \cref{selfdualperp}, (6) implies (7) by \cref{dualdefinable}, and (7) clearly implies (8). Also (8) implies (1) holds since definable subcategories are closed under products.
\end{proof}

\begin{rem}
In~\cite{BHS} a notion of stratification by the Balmer spectrum was developed which generalizes the notion of stratification described above. It is not immediate that one can prove a version of the previous theorem in this setup since the proof of the equivalence between (1)-(4) relies upon the existence of injective hulls of $R/\p$ for primes $\p$ in $\mrm{Spec}(R)$ which no longer exist in the setting of the Balmer spectrum. Instead, we expect that one can proceed by using the pure injective object determined by the associated homological residue field~\cite{BKS}, but treating this here would be outside the scope of this work. \end{rem}

We end this section with a couple of important examples of the previous theorem.
\begin{cor}\label{smashingselfdual}
	Let $\T$ be a big tensor-triangulated category which is canonically stratified. Let $L$ be a smashing localisation on $\T$. Then the category of $L$-locals $L\T$ is a definable subcategory with the property that $L\T = (L\T)^d$.
\end{cor}
\begin{proof}
	If $L$ is a smashing localisation, then $L\T$ is a localising $\otimes$-ideal which is closed under products. Therefore the claim follows from \cref{stratequiv}.
\end{proof}

\begin{rem}
	The category of $L$-locals for a smashing localisation $L$ always forms a definable subcategory~\cite[4.4]{krsmash}. The previous result shows that these definable subcategories are moreover self-dual under the additional assumption of being canonically stratified.
\end{rem}

\begin{cor}
Let $\T$ be a big tensor-triangulated category which is canonically stratified. If $\mc{D}$ is a definable triangulated $\otimes$-ideal in $\T$, then $\mc{D} = \mc{D}^d$.
\end{cor}
\begin{proof}
Definable triangulated $\otimes$-ideals are in particular localising $\otimes$-ideals, so this follows from \cref{stratequiv}.
\end{proof}

\begin{rem}
The previous corollary can also be deduced from \cref{smashingselfdual} using that any definable triangulated $\otimes$-ideal is of the form $\mc{S}^{\perp_{\mathbb{Z}}}$ for $\mc{S}$ a smashing $\otimes$-ideal~\cite[5.2.13]{wagstaffe}.
\end{rem}

\section{Duality pairs induced by functors}\label{dpinducedbyfunctors}
In this section we investigate how certain functorial operations enable the transfer of duality pairs. We approach this from two angles. The first is by considering how on AGJ duality triples one can induce duality pairs on functor categories, and then investigating the relationship between duality pairs on the AGJ duality triple, the functor categories and $\ab$. The second approach is to consider how duality pairs in big tt-categories behave with respect to smashing localisations.

\subsection{Duality pairs on functor categories}
Throughout this section we fix an AGJ duality triple $(\T,\U,Q,D)$. 
For any functor $f\in\Mod{\T^{\c}}$, there is an induced functor $f^{*}\in (\U^{\c},\ab)$, given by $f^{*}(C)=f(DC)$ for any compact $C\in\U^{\c}$. Associated to $f^{*}$ is the nerve functor 
\begin{equation}\label{nervedef}
N_{f^{*}}(-):\ab\to\Mod{\U^{\c}}, \quad A\mapsto\Hom_{\Z}(f^{*}(-),A),
\end{equation}
as detailed in \cite{nlab:nerve_and_realization}. By fixing an abelian group $A$, the assignment $f\mapsto N_{f^{*}}(A)$ gives a functor $\Mod{\T^{\c}}\to\Mod{\U^{\c}}$. In the case that $A=\qz$, we denote this functor by $\partial$. By symmetry, we may also assume that $\partial\colon\Mod{\U^{\c}}\to\Mod{\T^{\c}}$ and, by definition, we see that for every $f\in\Mod{\T^{\c}}$, there is an isomorphism $\partial^{2}f\simeq f(-)^{++}$. Also note that $\partial$ sends coproducts to products. Consequently $(\Mod{\T^{\c}},\Mod{\U^{\c}},\partial)$ gives a duality triple of Grothendieck abelian categories.

We firstly record the following lemma.
\begin{lem}\label{partialyoneda}
Let $X \in \T$. Then $\partial\bm{y}X\simeq \bm{y}QX$.
\end{lem}
\begin{proof}
For any compact $C\in\U^{\c}$, there are isomorphisms
\begin{align*}
(\partial\bm{y}X)(C)&=\Hom_{\Z}(\Hom_{\T}(DC,X),\qz) \mbox{ by }\cref{nervedef,}\\
&\simeq\Hom_{\U}(D^{2}C,QX) \mbox{ by \cref{AGJdualitytriple},}\\
&\simeq \Hom_{\U}(C,QX).
\end{align*}
Hence $\partial\bm{y}X\simeq \bm{y}QX$. 
\end{proof}

We now show that the duality triple $(\Mod{\T^{\c}},\Mod{\U^{\c}},\partial)$ completely determines all symmetric and definable duality pairs on the AGJ duality triple $(\T,\U,Q,D)$. Given any class $\A$ in $\Mod{\T^\c}$ we define $\A_{0}=\{X\in\T:\bm{y}X\in\A\} \subseteq \T$, and similarly for classes in $\Mod{\U^\c}$.

\begin{thm}\label{symmlift}
Let $(\X,\Y)$ be a symmetric duality pair on an AGJ duality triple $(\T,\U,Q,D)$. Then there is a symmetric duality pair $(\A,\B)$ on $(\Mod{\T^{\c}},\Mod{\U^{\c}},\partial)$ such that $\A_{0}=\X$ and $\B_{0}=\Y$.
\end{thm}

\begin{proof}
Define the class $\A_{\Y}=\{F\in\Mod{\T^{\c}}:\partial F\in\bm{y}\Y\}$, and likewise define $\A_{\X}\subseteq\Mod{\U^{\c}}$. First we show that $\A_{\Y}$ is closed under pure subobjects. 

Recall from~\cite[2.7]{krsmash} that a functor is flat if and only if it is cohomological. If $M \in \Mod{\T^\c}$ is cohomological, then $\partial M = \Hom_\Z(M(D(-)), \qz)$ is also cohomological. Suppose that $G\to F$ is a pure monomorphism with $F\in\A_{\Y}$. Then $\partial F\to\partial G$ splits in $\Mod{\U^{\c}}$, and $\partial F\simeq\bm{y}Y$ for some $Y\in\Y$ by definition. Therefore $\partial F$ is a flat functor. Since $\T^{\c}$ and $\U^{\c}$ are idempotent complete and closed under weak kernels and cokernels, the categories $\msf{Flat}(\T^{\c})$ and $\msf{Flat}(\U^{\c})$ are definable~\cite[6.1(a)]{dac}. In particular, since $\partial F$ is flat so is $\partial^{2}F$, and hence both $F$ and $G$ are flat as they are pure subobjects of $\partial^{2}F$. Consequently $\partial G$ is flat and pure injective, hence injective (this can be deduced from \cite[2.7]{krsmash} and \cite[5.6]{dac}). Thus, by \cite[1.9]{krsmash} we have $\partial G\simeq \bm{y}U$ for some pure injective $U\in\U$. Therefore $\bm{y}U$ is a summand of $\bm{y}Y \simeq \partial F$, so, again by \cite[1.7]{krsmash}, we deduce that $U$ is a summand of $Y$. But since $\Y$ is closed under summands, it follows $U\in\Y$ and, by definition $G\in\A_{\Y}$, which proves the claim. 

We now claim that $(\A_{\X},\A_{\Y})$ is a symmetric duality pair. Suppose that $F\in\A_{\X}$, so by definition, $\partial F\in\bm{y}\X$ so $\partial F \simeq \bm{y}X$ for some $X \in \X$. By~\cref{partialyoneda}, $\partial^{2}F\simeq \partial\bm{y}X\simeq \bm{y}QX\in\bm{y}Y$ as $(\X,\Y)$ is a duality pair. Thus $\partial F\in\A_{\Y}$, by definition. Conversely, suppose $\partial G\in\A_{\Y}$, then $\partial^{2}G\in\bm{y}\Y$ by definition. As $(\X,\Y)$ is symmetric, we have $\bm{y}\Y \subseteq \A_\X$, so $\partial^2G \in \A_\X$. Yet the first proven claim showed that $\A_{\X}$ is closed under pure subobjects, hence $G\in\A_{\X}$. It also follows from the first claim that $\A_{\Y}$ is closed under summands, and it is trivially closed under finite coproducts as $\Y$ is. Consequently $(\A_{\X},\A_{\Y})$ is a duality pair, and by a symmetric argument we see that it is actually symmetric.

Lastly, we claim that $\X=(\A_{\Y})_{0}$. Suppose $T\in(\A_\Y)_0$ so that $\bm{y}T\in\A_{\Y}$. Then $\partial \bm{y}T\simeq\bm{y}QT\in\bm{y}\Y$, hence $QT\in\Y$ and thus $T\in\X$. The other inclusion is clear: if $X\in \X$ then $\partial\bm{y}X \simeq \bm{y}QX \in\bm{y}\Y$, so $\bm{y}X\in\A_{\Y}$. Symmetrically, $\Y = (\A_\X)_0$ which finishes the proof. 
\end{proof}

As a corollary, we can specialise to the case where all the classes under consideration are definable.
\begin{cor}\label{defsymlift}
If $\X$ is a definable subcategory of $\T$, then $\A_\X$ is a definable subcategory of $\Mod{\U^\c}$. Consequently if $(\X,\Y)$ is a duality pair on $(\T,\U,Q)$ where both $\X$ and $\Y$ are definable, then there is a symmetric duality pair $(\A_\Y, \A_\X)$ on $(\Mod{\T^{\c}},\Mod{\U^{\c}},\partial)$ with $(\A_{\Y})_{0}=\X$ and $(\A_{\X})_{0}=\Y$.
\end{cor}
\begin{proof}
We have a symmetric duality pair $(\X,\X^d)$ by \cref{dualdefinable}. By \cref{symmlift}, $(\A_{\X},\A_{\X^d})$ is a symmetric duality pair, so by \cref{sym} it suffices to show that $\A_\X$ is closed under coproducts. Suppose that $\{F_{i}\}_{I}\subset\A_{\X}$ is a set of functors with $\partial F_{i}\simeq\bm{y}U_{i}$, with $U_{i}\in\X$. Then $\partial(\oplus_{I} F_{i})\simeq \prod\partial F_{i}\simeq \prod_{i}\bm{y}U_{i}\simeq \bm{y}(\prod_{I}U_{i})$. Yet as $\X$ is assumed to be definable, it is clear that $\bm{y}(\prod_{I}U_{i})\in\bm{y}\X$, hence $\oplus_{I}F_{i}\in\A_{\X}$. The rest follows immediately from \cref{symmlift}.
\end{proof}

Having shown that every symmetric (or definable) duality pair on $(\T,\U,Q,D)$ can be realised as the representable objects of a symmetric (or definable) duality pair on $(\Mod{\T^{\c}},\Mod{\U^{\c}},\partial)$, one may wonder whether we can deduce that every duality pair on $(\T,\U,Q,D)$ can be realised this way. There are immediate obstacles for this: firstly, unlike in the symmetric case, in a general duality pair $(\A,\B)$ the class $\B$ need have very few closure properties, particularly in relation to the pure structure (such as not being closed under pure subobjects, or pure injective envelopes). Nor is $\B$ unique; as previously mentioned, it is only unique up to dual modules. There is a further issue which one encounters: the objects in $\B$ need not be pure-injective, or pure submodules of a pure-injective in $\B$ (the latter is the case in the symmetric situation since $X \to \partial^2 X$ is a pure monomorphism). This means that, were one to consider the class $\bm{y}\B\in\Mod{\U^{\c}}$, one need not even obtain a class that is closed under direct summands. 

However, we can say something about duality classes (i.e., the left hand classes of duality pairs), rather than duality pairs themselves. 

\begin{prop}\label{dplift}
Let $(\T,\U,Q,D)$ be an AGJ duality triple. Let $(\A,\B)$ be a duality pair on the induced duality triple $(\Mod{\T^{\c}},\Mod{\U^{\c}},\partial)$. Then $(\A_{0},\B_{0})$ is a duality pair on $(\T,\U,Q,D)$. Moreover, if $(\X,\Y)$ is a duality pair on $(\T,\U,Q)$, then there is a duality pair $(\A,\B)$ on $(\Mod{\T^{\c}},\Mod{\U^{\c}},\partial)$ such that $\A_{0}=\X$ and $\B_{0} = \msf{add}(Q\X) \subseteq \Y$.
\end{prop}

\begin{proof}
For the first claim we have $X\in\A_{0}$ if and only if $\bm{y}X\in\A$ if and only if $\partial\bm{y}X\in\B$. By \cref{partialyoneda}, $\partial\bm{y}X \simeq \bm{y}QX$, so this is also equivalent to $QX\in\B_{0}$ as required. It is trivial that $\B_{0}$ is closed under finite direct sums and summands.

For the second claim, suppose $(\X,\Y)$ is a duality pair on $(\T,\U,Q)$ and consider the class $\bm{y}\X\subset\Mod{\T^{\c}}$. Let $(\overline{\bm{y}\X},\msf{add}(\partial\bm{y}\X))$ denote the minimal duality pair generated by $\bm{y}\X$, as in \cref{gendp}. We now show that $(\overline{\bm{y}\X})_0 = \X$. The inclusion $\X \subseteq (\overline{\bm{y}\X})_0$ is clear. So suppose that $M\in\T$ is such that $\bm{y}M\in\overline{\bm{y}\X}$. By \cref{partialyoneda} and the fact that $\X$ is closed under finite sums, $\bm{y}Q\X \simeq \partial\bm{y}\X$ is closed under finite sums, and hence $\mrm{add}(\partial\bm{y}\X)$ consists of retracts of elements of $\bm{y}Q\X$. Then, by definition, $\partial\bm{y}M\simeq \bm{y}QM$ is a summand of $\bm{y}QX$ for some $X\in \X$. Since $QX$ is pure injective, it follows from~\cite[1.7]{krsmash} that $QM$ is a summand of $QX$. Hence $QM \in \Y$ and therefore $M \in \X$ as required. 

To show that $\mrm{add}(\partial\bm{y}\X)_{0}\subseteq\Y$, note that as $QX$ is pure injective, $\partial\bm{y}X\simeq\bm{y}QX$ is injective~\cite[1.9]{krsmash}, so every functor in $\mrm{add}(\partial\bm{y}\X)$ is also injective. Hence every object of $\mrm{add}(\partial\bm{y}\X)$ is of the form $\bm{y}U$ for a pure injective object $U\in\U$. Another application of \cite[1.7]{krsmash} shows that $U$ must be a summand of an object of the form $QX$, and thus is an object in $\Y$.   
\end{proof} 

So far, we only used the nerve construction to define the functor $\partial$. We now use the fact that it has a left adjoint to relate duality pairs on $(\Mod{\T^{\c}},\Mod{\U^{\c}},\partial)$ to those on $(\ab,\ab,(-)^{+})$. For any $f\in\Mod{\T^{\c}}$, we have defined $\partial f=N_{f^{*}}(\qz)$. The left adjoint to the nerve functor $N_{f^{*}}(-)$ is the realisation functor $\vert-\vert_{f^{*}}:\Mod{\U^{\c}}\to\ab$ defined as the the left Kan extension of $f^{*}$ along the restricted Yoneda embedding $\bm{y}$. There is then, for any functor $g\in\Mod{\U^{\c}}$, an isomorphism
\[
\Hom_{\Z}(\vert g\vert_{f^{*}},\qz)\simeq\Hom_{\Mod{\U^\c}}(g,N_{f^{*}}(\qz)).
\]
By definition of the realisation as the left Kan extension, one sees that $\vert \bm{y}DC\vert_{f^{*}}\simeq f^{*}(DC)\simeq f(C)\simeq\Hom(\bm{y}C,f)$. Therefore, by taking $g = \bm{y}DC$ in the above adjunction, we obtain the following lemma.
\begin{lem}
Let $(\T,\U,Q,D)$ be an AGJ duality triple. Then for any compact $C\in\T^{\c}$, the square
\[
\begin{tikzcd}\label{nrdiagram}
\Mod{\T^{\c}} \arrow[r, "\partial"] \arrow[d, swap, "{\Hom(\bm{y}C,-)}"]& \Mod{\U^{\c}} \arrow[d, "{\Hom(\bm{y}DC,-)}"] \\
\ab  \arrow[r, swap,"(-)^{+}"] & \ab
\end{tikzcd}
\]
commutes.
\end{lem}
The following result which shows how one can lift duality pairs from $(\ab,\ab,(-)^{+})$ to $(\Mod{\T^{\c}},\Mod{\U^{\c}},\partial)$, is essentially a corollary of \cref{nrdiagram}.

\begin{cor}\label{abtransfer}
Let $(\A,\B)$ be a duality pair in $\ab$. Fix a compact $C\in\T^{\c}$, and set $\X=\{F\in\Mod{\T^{\c}}:\Hom(\bm{y}C,F)\in\A\}$ and $\Y=\{G\in\Mod{\U^{\c}}:\Hom(\bm{y}DC,G)\in\B\}$. Then $(\X,\Y)$ is a duality pair on $(\Mod{\T^{\c}},\Mod{\U^{\c}},\partial)$. Moreover, if $(\A,\B)$ is symmetric (resp., coproduct closed, resp., product closed, resp., definable), then so is $(\X,\Y)$.
\end{cor}
\begin{proof}
By definition of $\X$ and the fact that $(\A,\B)$ is a duality pair,  we see that $F \in \X$ if and only if $\Hom(\bm{y}C,F)^+ \in \B$. From \cref{nrdiagram} we see that this is the case if and only if $\Hom(\bm{y}DC,\partial F)\in \B$, which by definition is equivalent to $\partial F \in \Y$. Since $\B$ is closed under finite direct sums and direct summands, and $\Hom(\bm{y}DC,-)$ is additive, we also see that $\Y$ is closed under finite direct sums and direct summands.
\end{proof}

Combining the previous corollary with~\cref{dplift} we obtain a way of transporting duality pairs from $(\ab,\ab,(-)^{+})$ to the triangulated category. The following example gives an illustration of this.

\begin{ex}\label{ex:moduledualitypair}
Let $R$ be a ring and $\mc{A}\subset\Mod{R}$ and $\mc{B}\subset\Mod{R^{\circ}}$ be such that $(\mc{A},\mc{B})$ form a duality pair. Then define
\[
\msf{H}_{n}^{-1}(\mc{A})=\{X\in\D(R):H_{n}(X)\in \mc{A}\}
\]
and similarly define $\msf{H}^{-1}_{n}(\mc{B})$. By taking $C=R$ in \cref{abtransfer} and applying the $(-)_0$ construction as in \cref{dplift}, we immediately see that 
$(\msf{H}_{n}^{-1}(\mc{A}),\msf{H}_{-n}^{-1}(\mc{B}))$ forms a duality pair in $\D(R)$ for each $n\in\mbb{Z}$.
In particular, as $(\mc{A},\mc{B})=(H_{n}(\msf{H}^{-1}_{n}(\mc{A})),H_{-n}(\msf{H}^{-1}_{-n}(\mc{B})))$,  every duality pair of modules arises as the homology of a duality pair of complexes.
\end{ex}

\subsection{Duality pairs in localisations}\label{localisations}
We now turn to investigating how certain localisations on big tt-categories respect duality pairs. Throughout, by a localisation we mean a monoidal Bousfield localisation as in~\cite[3.1.1(a)]{axiomatic}; that is, an exact functor $L\colon\T\to\T$ equipped with a natural transformation $\eta\colon\t{Id}\to L$ such that $L\eta$ is invertible, $L\eta=\eta L$, and $\mrm{ker}(L)$ is a $\otimes$-ideal. A localisation is called \ti{smashing} if it preserves coproducts, or equivalently if the natural map $L\1 \otimes X \to LX$ is an isomorphism for all $X \in \T$. Examples of localisation functors abound. For example, in $\msf{D}(R)$ the functor $M \mapsto M_\p$ for any prime ideal $\p \in \mrm{Spec}(R)$ is a smashing localisation functor, and the derived completion functor $\Lambda_\p$ is a localisation functor but it is not smashing in general. Colocalisation functors are defined dually.

Let $L$ be a smashing localisation of a big tt-category $\T$. The smashing localisation $L$ has an associated colocalisation $\Gamma$ determined by the triangle $\Gamma X \to X \to LX$ for all $X \in \T$. Moreover, $\Gamma$ is also smashing in the sense that $\Gamma X \simeq \Gamma\1 \otimes X$ for all $X \in \T$. We then write $V = F(L\1,-)$ and $\Lambda = F(\Gamma\1, -)$, and note that $V$ is a colocalisation and $\Lambda$ is a localisation, but neither of these are smashing in general. We emphasize here that the images of $L$ and $V$ are the same, but that they will in general be different objectwise.

Since $L$ is smashing, we note that the local category $L\T$ is again a big tt-category with monoidal product and internal hom the same as those in $\T$ and monoidal unit $L\1$. Firstly, we identify the Brown--Comenetz object in the local category $L\T$. We write $\mbb{I}^L$ for the Brown--Comenetz object in $L\T$.
\begin{lem}\label{lem:localBC}
	We have $\mbb{I}^L \simeq V\mbb{I}$ and hence $\mbb{I}^L(LX) \simeq V(\mbb{I}X)$ for all $X \in \T$. Moreover, if $X$ is $L$-local, then $\mbb{I}^L(X) \simeq \mbb{I}X$.
\end{lem}
\begin{proof}
	For the first part it is sufficient to show that $V\mbb{I}$ has the defining property of $\mbb{I}^L$. We have \[\Hom_{L\T}(Y, V\mbb{I}) \simeq \Hom_{\T}(Y, V\mbb{I}) \simeq \Hom_\T(Y, \mbb{I})\] which proves the first claim. The second claim follows since \[\mbb{I}^L(LX) \simeq F(LX, V\mbb{I}) \simeq F(LX, \mbb{I}) \simeq F(L\1, \mbb{I}X) \simeq V\mbb{I}X\] by adjunction and the fact that $L$ is smashing. We now prove the final claim. Since $X$ is $L$-local, we have $X \simeq LY$ for some $Y \in \T$. Therefore $\mbb{I}^L(X) \simeq \mbb{I}^L(LY) \simeq V(\mathbb{I}Y)$ by \cref{lem:localBC}. Now \[V(\mathbb{I}Y) \simeq F(L\1, \mbb{I}Y) \simeq F(LY, \mbb{I}) \simeq \mbb{I}X\] by adjunction. 
\end{proof}

Firstly we show that duality pairs on $\T$ give rise to duality pairs on the local category in the obvious way.
\begin{prop}\label{prop:localizeDP}
	Let $(\A, \B)$ be a duality pair on $(\T, \T, \mbb{I})$. Then the assignment $\Phi\colon (\A,\B) \mapsto (L\A, L\B)$ yields a duality pair on $(L\T, L\T, \mbb{I}^L)$. Moreover, if $(\A,\B)$ is a symmetric duality pair, then $(L\A, L\B)$ is also symmetric.
\end{prop}
\begin{proof}
	Observe that $X \in L\A$ if and only if $X$ is $L$-local and $X \in \A$, and similarly for $L\B$. So if $X$ is $L$-local and $X \in \A$, then $\mbb{I}^LX \simeq \mbb{I}X \in \B$ (using \cref{lem:localBC}) and since $\mbb{I}^L(X)$ is $L$-local, we have $\mbb{I}^L(X) \in L\B$. Conversely, let $X$ be $L$-local and suppose that $\mbb{I}^L(X) \in L\B$. By \cref{lem:localBC}, we have $\mbb{I}^L(X) \simeq \mbb{I}X \in \B$ and hence $X \in \A$, and thus in $L\A$ since it is $L$-local. The closure of $L\B$ under retracts and finite sums is clear. The claim about symmetry is clear.
\end{proof}

Now we turn to the converse problem; namely, we construct lifts of duality pairs from $L\T$ to $\T$. 
\begin{prop}\label{prop:liftDPfromlocals}
	Let $(\A,\B)$ be a duality pair on $(L\T, L\T, \mbb{I}^L)$. Define \[{}^\uparrow\msf{A} = \{X \in \T \mid LX \in \A\} \quad \mrm{and} \quad \msf{B}^\uparrow = \{X \in \T \mid VX \in \B\}.\] Then the assignment $\Psi \colon (\A,\B) \mapsto ({}^\uparrow\A, \B^\uparrow)$ gives rise to a duality pair on $(\T, \T, \mbb{I})$.
\end{prop}
\begin{proof}
	Let $X \in {}^\uparrow\A$ so that $LX \in \A$. Then $V\mbb{I}X \simeq \mbb{I}^L(LX) \in \B$ using \cref{lem:localBC}. Conversely, if $\mbb{I}X \in \B^\uparrow$, then $\mbb{I}^L(LX) \simeq V\mbb{I}X \in \B$ so $LX \in \A$ as required. Now $L$ and $V$ preserve finite sums so the required closure properties are trivial.
\end{proof}

\begin{thm}\label{surjdp}
	The composition $\Phi \circ \Psi$ is the identity on duality pairs in $L\T$. In particular, every duality pair on $L\T$ arises from a duality pair on $\T$ via $\Phi$.
\end{thm}
\begin{proof}
	It is immediate from the definitions that $(\Phi \circ \Psi)(\A,\B) = (L(\,{}^\uparrow\A), L(\B^\uparrow)) = (\A,\B)$. Since $L$ is idempotent it is clear that $L(\,{}^\uparrow\A) = \A$. For the $\B$ class, we have $X \in L(\B^\uparrow)$ if and only if $X$ is $L$-local and in $\B^\uparrow$. Since $X$ is $L$-local, $X \simeq VZ$ for some $Z$, and hence $X \simeq VZ \simeq V^2Z \simeq VX \in \B$. 
\end{proof}

\begin{cor}\label{ttselfdual}
	Let $\mc{D}$ be a definable subcategory of a big tt-category $\T$. Then $L\mc{D}$ is a definable subcategory of $L\T$; moreover, we have that $(L\mc{D})^{d}=L(\mc{D}^{d})$.
\end{cor}
\begin{proof}
	By \cref{dualdefinable}, we have a symmetric duality pair $(\mc{D}, \mc{D}^d)$ and $\mc{D}$ is closed under coproducts. By \cref{prop:localizeDP}, we have a symmetric duality pair $(L\mc{D}, L(\mc{D}^d))$ on $L\T$ and since $L$ preserves arbitrary coproducts as it is smashing, we have that $L\mc{D}$ is coproduct closed. Applying \cref{definablett} shows that $L\mc{D}$ is a definable subcategory of $L\T$ and that $(L\mc{D})^d = L(\mc{D}^d)$.
\end{proof}

On the other hand, we can also consider the assignment $\Psi\circ \Phi$, which sends a duality pair on $\T$ to another duality pair on $\T$. Explicitly,
\[
(\Psi\circ\Phi)(\A,\B)=(\{X\in\T:LX\in L\A\},\{X\in\T:VX\in L\B\})
\]
it is clear that this is not the identity on duality pairs, as other things may localise into $L\A$. However, the operation $\Psi\circ\Phi$ is still well behaved, as the following shows.

\begin{prop}
The operation $\Psi\circ\Phi$ is idempotent; that is, if $(\A,\B)$ is a duality pair, we have
\[
(\Psi\circ\Phi)^{2}(\A,\B)=(\Psi\circ\Phi)(\A,\B).
\]
\end{prop}

\begin{proof}
From the above description, we see that the left class in $(\Psi\circ\Phi)^{2}(\A,\B)$ is exactly
\[
\{X\in\T:LX\in\{Z\in\T:LZ\in L\A\}\};
\]
in particular, if $X$ is in this class, then $L(LX)\in L\A$, but as $L^{2}=L$ we have $LX\in L\A$, so $X$ is, in fact in the left class of $(\Psi\circ\Phi)(\A,\B)$. By inspection, it is clear that $X\in\T$ is in the right class of $(\Psi\circ\Phi)^{2}(\A,\B)$ if and only if $V^{2}X\in L\B$. However $V^2X \simeq VX$ and thus $VX\in L\B$, hence $X$ is in the right class of $(\Psi\circ\Phi)(\A,\B)$.
\end{proof}

\bibliographystyle{abbrv}

\bibliography{references}

\end{document}